\documentclass[11pt]{article}

\usepackage[utf8]{inputenc}
\usepackage{blindtext}
\usepackage{mathtools}
\usepackage{amsmath}
\usepackage{geometry}
\usepackage{amsfonts}
\usepackage{dsfont}
\usepackage{filecontents}
\usepackage[justification=centering]{caption}
\usepackage{diagbox}
\usepackage{graphicx}
\usepackage{yhmath}
\usepackage{titlesec}
\usepackage{tocloft}
\usepackage{etoolbox}
\usepackage{amsthm}
\usepackage{amssymb}
\usepackage{subcaption}
\usepackage[scr=esstix,cal=boondox]{mathalfa} 
\usepackage{aligned-overset}
\usepackage[linktoc=all]{hyperref}
\hypersetup{
    colorlinks=true,
    linkcolor=blue,
    filecolor=black,      
    urlcolor=black,
    citecolor=blue,
    }
\usepackage{letltxmacro}
\usepackage[section,sort=standard]{glossaries}
\usepackage{glossary-mcols}
\usepackage{enumitem}

\makeatletter
  \newglossarystyle{mymcolalttree}{%
    \setglossarystyle{alttree}%
      {%
        \begin{multicols}{2}%
        \def\@gls@prevlevel{-1}%
        \mbox{}\par
      }%
      {\par\end{multicols}}%
      \columnsep 3em
  }%
\makeatother

\makeglossaries

\newcommand*{\newconstant}[5][]{%
  \newglossaryentry{#2}{name={$#3$},text={#3},description={\hspace{5em}#4\vspace{0.2em}},#1,sort={#5}}%
}

\titlespacing*{\section}{0em}{5em}{2em}
\titlespacing*{\subsection}{0em}{4em}{1em}
\titlespacing*{\subsubsection}{0em}{2em}{1em}

\renewcommand*{\thesection}{\arabic{section}.}
\renewcommand*{\thesubsection}{\alph{subsection}.}
\renewcommand*{\thesubsubsection}{\alph{subsubsection}.}

\titleformat{\section}
  {\normalfont\Large\bfseries}{\thesection}{0.15em}{}
\titleformat{\subsection}
  {\normalfont\Large\bfseries}{\thesubsection}{0.3em}{}
\titleformat{\subsubsection}
  {\normalfont\Large\bfseries}{\thesubsubsection}{0.45em}{}

\makeatletter
\renewcommand{\numberline}[1]{%
  \@cftbsnum #1\@cftasnum~\@cftasnumb%
}
\makeatother

\newtheorem{thm}{\hspace{1em} Theorem}[section]
\newtheorem{thmtitle}{\hspace{1em} Theorem}[section]
\newtheorem*{thm*}{\hspace{1em} Theorem}
\newtheorem*{conj*}{\hspace{1em} Viana's conjecture}
\newtheorem{definition}[thm]{\hspace{1em} Definition}
\newtheorem{lem}[thm]{\hspace{1em} Lemma}
\newtheorem{rmq}[thm]{\hspace{1em} Remark}
\newtheorem*{lem*}{\hspace{1em} Lemma}
\newtheorem*{rmq*}{\hspace{1em} \textit{Remark}}

\newtheorem{cortitle}[thmtitle]{\hspace{1em} Corollary}
\newtheorem{prop}[thm]{\hspace{1em} Proposition}
\newtheorem{proptitle}[thmtitle]{\hspace{1em} Proposition}

\newtheorem*{propriete*}{\hspace{1em} Property}

\makeatletter
\def\thm@space@setup{%
  \thm@preskip=1.5em \thm@postskip=2em
}
\makeatother

\let\oldautoref\autoref
\DeclareRobustCommand{\autoref}[1]{{\let\nobreakspace\space\oldautoref{#1}}}

\newcommand{\bbR}{{\mathbb{R}}}

\newcommand{\bbZ}{{\mathbb{Z}}}

\newcommand{\bbT}{{\mathbb{T}}}

\newconstant{Chi}{\chi(x)}{Lyapunov exponent at $x$}{11}
\newconstant{Basin}{\mathcal{B}(\mu)}{Basin of $\mu$}{5}
\newconstant{Rf}{R(f)}{$\lim_n \frac{1}{n} \log || (f^n)' ||_{\infty}$}{53}
\newconstant{Crit}{\mathcal{C}}{Set of critical points of $f$}{8}
\newconstant{Z_plus}{\bbZ^+}{$\{ 1 ; 2 ; ... \}$}{64}
\newconstant{Critn}{\mathcal{C}_n}{Set of critical points of $f^n$}{9}
\newconstant{Critinf}{\mathcal{C}_{\infty}}{Union of $\mathcal{C}_n$'s}{10}
\newconstant{Kng}{k_{n,g}(z)}{$\lfloor \log^+ |g'(g^{n-1}(z))|$}{33}
\newconstant{Kng_prime}{k_{n,g}'(z)}{$\lfloor \log^- |g'(g^{n-1}(z))|$}{34}
\newconstant{Zz_plus}{\bbZ_0^+}{$\{ 0 ; 1 ; 2 ; ... \}$}{65}
\newconstant{Hg}{\mathcal{H}_g(k,k')}{$\{ x \mid k_{i,g}^{(')}(x) = k_i^{(')}, i \in [\![ 1 ; n ]\!] \}$}{29}
\newconstant{Ep}{E_p(x)}{Geometric times of $x$ for $f^p$}{21}
\newconstant{D_upp}{\overline{d}(E)}{$\limsup_n d_n(E)$}{15}
\newconstant{D_part}{d_n(E)}{$\frac{1}{n} \# (E \cap [\![ 1 ; n ]\!])$}{13}
\newconstant{D_along}{d^{\mathcal{n}}(E)}{$\lim_{n \in \mathcal{n}} d_n(E)$}{14}
\newconstant{EnM}{E_n^M}{Defined as $\bigcup\limits_{\substack{n > k,l \in E\\
|k-l| \leq M}} [ \! [ k ; l [ \! [ \: \subset \bbZ_0^+$}{22}
\newconstant{EnMm}{E_n^{M,m}}{Subset of $E_n^M$}{23}
\newconstant{E_front}{\partial E}{Border of $E$, i.e. : $E \Delta (E+1)$}{16}
\newconstant{EnMmx}{E_n^{M,m}(x)}{$(E_p(x))_n^{M,m}$}{24}
\newconstant{A}{A}{$\{ \chi > R(f) / r \}$}{1}
\newconstant{B}{b}{$b > R(f)/r$ s.t. $Leb(\chi > b) > 0$}{3}
\newconstant{P}{p}{$p$ s.t. $E_p(x)$ has positive density}{45}
\newconstant{Beta}{\beta}{value such that $\overline{d}(E_p(x)) > \beta$}{6}
\newconstant{Eps}{\varepsilon}{Yomdin's scale for $f^p$}{18}
\newconstant{Sigma}{\sigma}{Reparametrization intersecting $A$}{57}
\newconstant{Ex}{E(x)}{$E_p(x)$ with $p$ large enough}{20}
\newconstant{G}{g}{$g = f^p$ with $p$ large enough}{27}
\newconstant{N}{\mathcal{n}}{Sequence of integers along which we
\hspace*{5em}have nice convergence properties}{40}
\newconstant{A_n}{A_n}{Subset of $A$ where $E_p(x)$ has\\
\hspace*{5em}uniformly bounded density}{2}
\newconstant{J}{\mathcal{J}}{Partition into monotone branches of
\hspace*{5em}$g$ of the form $[x  ; y[$}{32}
\newconstant{EnMm_ronde}{\mathcal{E}_n^{M,m}}{Partition into values of $E_n^{M,m}(x)$}{25}

\newconstant{MunM}{\mu_n^{M}}{$\mu_n^{M,m}$ for $m=1$}{38}
\newconstant{NunM}{\nu_n^{M}}{$(\beta_n^M / \beta^{\infty}) \times \mu_n^M$}{42}

\newconstant{MunMm}{\mu_n^{M,m}}{Empirical probabilty measure}{37}
\newconstant{NunMm}{\nu_n^{M,m}}{$(\beta_n^{M,m} / \beta^{\infty})  \times \mu_n^{M,m}$}{41}
\newconstant{MuM}{\mu^{M,m}}{$\lim_{n \in \mathcal{n}} \mu_n^M$}{39}
\newconstant{NuM}{\nu^{M,m}}{$\lim_{n \in \mathcal{n}} \nu_n^M$}{43}
\newconstant{Mu}{\mu}{$\mu = \lim_M \mu^M = \lim_M \nu^M$}{40}
\newconstant{Phig}{\phi_g}{$\log |g'|$}{48}
\newconstant{Partition_mu}{\mathcal{P}_{\mu}}{$\{ P \in \mathcal{P} \mid \mu(P) > 0 \}$}{46}
\newconstant{RF}{\mathcal{R}^F}{$\bigvee_{i \in F} T^{-i} \mathcal{R}$}{54}
\newconstant{Mu_E}{\mu_E}{Defined as $\frac{Leb_{A_n}( \; \cdot \; \cap E)}{Leb_{A_n}(E)}$}{36}
\newconstant{Chi_gmu}{\chi_g(\mu)}{$\int \chi_g \: d\mu$}{12}
\newconstant{Iqk}{I_{q,k}}{$]k/q ; k+1/q] +a$ for some fixed $a$}{31}
\newconstant{Qqk}{Q_{q,k}}{Defined by $(\log |g'|)^{-1}(I_{q,k})$}{51}
\newconstant{Qq}{\mathcal{Q}_q}{Partition given by $Q_{q,k}$'s}{50}
\newconstant{Pq}{\mathcal{P}_q}{$\mathcal{Q}_q \vee \mathcal{J}$}{47}
\newconstant{PhigE}{\phi_g^E}{$\sum_{k \in E} \phi \circ g^k$}{49}
\newconstant{E}{E}{Some fixed $E_n^M(x)$ of the form
\hspace*{5em}$E=\bigcup\limits_{j \in [\![ 1 ; m ] \! ]} [\![ a_j ; b_j [\![$}{19}
\newconstant{R}{R}{$\mathcal{J}^E_x \cap \mathcal{Q}_{q,x}^{E} \cap E \cap A_n \subset I$}{52}
\newconstant{Boules}{\mathcal{B}}{Partition of $I$ into small balls}{4}
\newconstant{Partition_m}{\mathcal{R}^m}{$\bigvee_{i \in [\![ 0 ; m [\![} T^{-i} \mathcal{R}$}{56}
\newconstant{Vaj}{\mathcal{V}_{a_{j+1}}}{$\mathcal{B}^{\{a_{j+1} - b_j \}} \vee \mathcal{J}^{a_{j+1}-b_j}$}{63}
\newconstant{V}{\mathcal{V}}{$\bigvee\limits_{j=0}^{m-1} g^{-b_j}\mathcal{V}_{a_{j+1}}$}{62}
\newconstant{HACIP}{HACIP}{Hyperbolic Absolutely Continuous
\hspace*{5em}Invariant Probabilty}{28}
\newconstant{Nu}{\nu}{$\nu = \frac{1}{p}\sum\limits_{k=0}^{p-1} f^k_* \mu$}{44}
\newconstant{BetanMm}{\beta_n^{M,m}}{$\int d_n(E^{M,m}(x)) d Leb_{A_n}(x)$}{7}
\newconstant{iter}{f^k}{$f$ iterated $k$ times}{26}
\newconstant{derivordre}{d^k f}{$k$-th order derivative of $f$}{17}
\newconstant{part_entiere}{\lfloor x \rfloor}{integer part of $x$}{30}
\newconstant{Leb_norm}{Leb_J}{$Leb( \cdot \cap J) / Leb(J)$}{35}
\newconstant{tree_n}{\mathcal{T}_n}{$n$-th level of the tree of
\hspace*{5em}reparametrizations}{58}
\newconstant{tree_n_geom}{\overline{\mathcal{T}_n}}{Reparametrizations of $\mathcal{T}_n$ that see
\hspace*{5em}the expansion of $f$}{59}
\newconstant{theta_n}{\theta_{i^n}}{Reprametrization at level $n$}{60}
\newconstant{theta_nk}{\theta_{i^n_k}}{Child of $\theta_{i^n}$ at level $k$}{61}

\counterwithin{equation}{section}

\title{Hyperbolic absolutely continuous invariant measures for $\mathcal{C}^r$ one-dimensional maps}
\date{}
\author{Alexandre Delplanque\\\\ \textit{LPSM, Sorbonne Universit\'e, Paris, France}}

\begin{document}

\maketitle

\textsc{Abstract:}
For $r > 1$, we show, using the Ledrappier-Young entropy characterization of SRB measures for non-invertible maps, that if a $\mathcal{C}^r$ map $f$ of the interval or the circle has its Lyapunov exponent greater than $\frac{1}{r} \log || f' ||_{\infty}$ on a set $E$ of positive Lebesgue measure, then it admits hyperbolic ergodic invariant measures that are absolutely continuous with respect to the Lebesgue measure.
We also show that the basins of these measures cover $E$ Lebesgue-almost everywhere.

\vspace{3em}
\begin{center}
\textsc{\Large Contents}\vspace{-4em}
\end{center}

\renewcommand*\contentsname{}
\tableofcontents


\section{Introduction}
\label{sec:intro}






In this work, we study the long-term behavior of discrete-time dynamical systems.
More precisely, we consider maps $f : X \to X$ where $X$ is the phase space,
and we define the \textit{orbit} of a point $x$ in $X$ as the sequence obtained by iterating $f$ starting from $x$.
We are interested in the asymptotic distribution of these orbits, for example by identifying attractors, their structure, and the set of points that are attracted to them.
Formally, we consider the notion of empirical measure: for a positive integer $n$ and a point $x$ in $X$, we define $\mu_n^x := \frac{1}{n} \sum\limits_{k=0}^{n-1} \delta_{f^k(x)}$, where $\delta_{f^k(x)}$ is the Dirac at $f^k(x)$ and $\gls*{iter}$ is the $k$-th iterate of $f$.
Therefore, the asymptotic behavior of the system can be understood by studying the limits of the sequence $(\mu_n^x)_{n}$ as $n$ goes to infinity.
By duality, if $\mu$ is an $f$-invariant borelian probability measure on $X$, we define the \textit{basin} of $\mu$, denoted by $\gls*{Basin}$, as the set of points $x$ such that $\left ( \mu_n^x \right )_n$ converges to $\mu$ for the weak-$*$ topology.
When $X$ is a smooth Riemannian manifold, one can define a Lebesgue measure on $X$ and search for measures whose basin has positive Lebesgue measure. They are called \textit{physical measures} and describe the asymptotic dynamics of a visible part of the space. As a consequence of Birkhoff's ergodic theorem, $f$-invariant ergodic probability measures that are absolutely continuous with respect to the Lebesgue measure are physical measures.
Our goal here is to study, for smooth one-dimensional dynamics, the existence of absolutely continuous probability measures of positive entropy and to understand their basins.\\

For one-dimensional dynamics, the existence of such measures has been studied since the 1970s, 
by Jakobson for example, for the quadratic family \cite{Jakobson1980} and for near-quadratic families \cite{Jakobson1981}, and by Collet and Eckmann \cite{ColletEckmann} for more general unimodal maps.
As for multimodal maps, some hyperbolicity assumption was usually made, such as in the results of Keller \cite{Keller_1990} and Ledrappier \cite{Ledrappier_1981}, which we will discuss later in the introduction.
Regarding higher-dimensional dynamics, the notion of SRB measures generalizes what absolutely continuous probability measures represent for one-dimensional systems. These measures first appeared in the works of Sinai \cite{Sinai_1972}, Ruelle \cite{Ruelle_1978} \cite{Ruelle_1976} and Bowen \cite{Bowen_1975} who showed their existence for uniformly hyperbolic diffeomorphisms.
The aim has then been to weaken the hyperbolicity assumption, both in dimension one and in higher dimension.
To do so, we use the notion of Lyapunov exponent, which we only define when $X=I$ is one-dimensional: for $x \in I$, the \textit{Lyapunov exponent} of $f$ at $x$ is 
$$
\gls*{Chi} = \limsup\limits_{n \to +\infty} \frac{1}{n} \log |(f^n)'(x)|
$$
Therefore, weakening the hyperbolicity assumption is well illustrated by Viana's conjecture \cite{Viana1998DynamicsAP}:

\begin{conj*}
If a smooth map has only non-zero Lyapunov exponents at Lebesgue almost every point, then it admits some SRB measure.
\end{conj*}

We present some previous results in this direction. For one-dimensional dynamics, and priorly to the conjecture, Keller \cite{Keller_1990} proved the existence of absolutely continuous measures for multimodal maps with negative Schwarzian derivative and with positive Lyapunov exponent on a set of positive Lebesgue measure.
Another similar result is that of Ledrappier \cite{Ledrappier_1981} where he showed that such measures exist for $\mathcal{C}^2$ maps satisfying the following conditions:\\[-2em]
\begin{itemize}
\item[i)] Non-degenerate critical points: there are finitely many critical points, all of finite multiplicity.\\[-2em]
\item[ii)] Hyperbolicity: on a set of positive Lebesgue measure, the Lyapunov exponent is positive.\\[-2em]
\item[iii)] Regularity assumption on the Lyapunov exponents: on this set of positive Lebesgue measure, the $\limsup$ defining the exponent is a limit and the orbits' distributions converge to an ergodic measure.\\[-2em]
\end{itemize}
The proof of Ledrappier uses an entropy characterization of absolutely continuous measures: among the ergodic measures of positive entropy, the absolutely continuous ones are exactly those who satisfy the entropy formula $h(\mu) = \int \log |f'| d \mu$.
As for higher-dimensional dynamics, the existence of SRB measures has been shown under similar though weaker conditions by Alves, Bonatti and Viana \cite{Alves2000}.
Moreover, their proof relies on a geometric study of the probability density with respect to the Lebesgue measure of the empirical measures rather than on the Ledrappier-Young entropy characterization \cite{Ledrappier_Young_1985}.\\

In this work we deal with interval or circle maps and we only assume some hyberbolicity and regularity of the map.
Hence, the dynamics is free to display flat critical points and infinitely many monotones branches.\\

We now introduce the required definitions and notations to state our result.
Let $I$ be the interval $[0 ; 1]$ or the circle $\bbT^1$ and $f : I \to I$ be a $\mathcal{C}^r$ map where $r > 1$. Here, $r$ does not need to be an integer, so $f$ being a $\mathcal{C}^r$ map means that it is $\mathcal{C}^{\lfloor r \rfloor}$, where $\lfloor r \rfloor$ is the largest integer smaller or equal to $r$, and that $d^{\lfloor r \rfloor} f$, the $\lfloor r \rfloor$-th order derivative of $f$, is Hölderian with exponent $r - \lfloor r \rfloor$, and we denote the Hölder constant by $|| d^r f ||_{\infty}$.
We also define
$\gls*{Rf} \overset{\text{def}}{=} \lim\limits_{n \to +\infty} \frac{1}{n} \log^+ ||(f^n)'||_{\infty} = || \chi^+ ||_{\infty} \leq \log || f' ||_{\infty}$ where $|| \cdot ||_{\infty}$ is the essential supremum norm.
Lastly, we say that an $f$-invariant probabilty $\mu$ is \textit{hyperbolic} if, for $\mu$-almost every $x$, we have $\chi(x) > 0$.

\begin{thmtitle}
\label{thm:thm}
Let $f : I \to I$ be a $\mathcal{C}^r$ map with $r > 1$.
There are countably many hyperbolic ergodic $f$-invariant measures $(\mu_i)$ that are absolutely continuous with respect to the Lebesgue measure and such that:
\begin{itemize}
\item[•] Lebesgue-almost every $x \in I$ such that $\chi(x) > \frac{R(f)}{r}$ is in some $\mathcal{B}(\mu_i)$ and $\chi(x) = \chi(\mu_i)$,\\[-1em]
\item[•] for any $\delta > 0$, the set $\left \{ \chi > \frac{R(f)}{r} + \delta \right \}$ is covered by finitely many of these basins, up to a set of zero Lebesgue measure.
\end{itemize}
In particular, for any $\delta > 0$, there are finitely many ergodic absolutely continuous measures with entropy larger than $\frac{R(f)}{r} + \delta$.
\end{thmtitle}

For smooth interval maps, we get the following corollary:

\begin{cortitle}
\label{cor:title}
Let $f : I \to I$ be a $\mathcal{C}^{\infty}$ map.
Then $f$ admits an absolutely continuous hyperbolic measure if and only if $Leb(\chi > 0) > 0$.
\end{cortitle}

Furthermore, regarding the finiteness statement in Theorem \autoref{thm:thm}, we explicit a bound on the number of measures. 
Such bounds are not known for surface diffeomorphisms, where Buzzi, Crovisier and Sarig's methods \cite{BCS_mme} only provide finiteness of SRB measures. Moreover, in the case where $f$ is analytic, the bound we obtain has a simple expression (see Remark \autoref{rmq:analytic}).
For the following statement, we introduce an additional notation: if $f$ is a $\mathcal{C}^r$ map, then $|| f' ||_{r-1}$ denotes $\max\limits_{k \in [\![ 1 ; \lfloor r \rfloor ]\!] \cup \{ r \}} || d^k f ||_{\infty}$ where $\gls*{derivordre}$ is the $k$-th order derivative of $f$.

\begin{proptitle}
\label{prop:title}
Let $f : I \to I$ be a $\mathcal{C}^r$ map with $r > 1$.
There exists a constant $C_r$ such that, for any $\delta > 0$, the number of hyperbolic ergodic $f$-invariant absolutely continuous measures whose basin intersects $\left \{ \chi > \frac{\log || f' ||_{\infty}}{r} + \delta \right \}$ with positive Lebesgue measure is less than
$
\left ( \frac{\log || f' ||_{\infty}}{\delta}  \right )^{\left ( C_r \log || f' ||_{r-1} \right ) / \delta}
$.
\end{proptitle}

We mention that if we also assume $f$ to be transitive in the previous statements, then there is at most one hyperbolic absolutely continuous measure. In particular, whenever such a measure exists, it is unique.
We further explain this fact in Remark \autoref{rmq:trans}.\\

We also mention that the bound $R(f) /r$ in Theorem \autoref{thm:thm} is sharp.
In Appendix A of \cite{BurguetCMP}, there is, for any $r \in \bbZ^+$, an example of a $\mathcal{C}^r$ map $f$ of the interval for which there exists some set $E \subset I$ satisfying 
\begin{itemize}
\item[-] $Leb(E) > 0$,
\item[-] for $x \in E$, $\chi(x) = R(f) / r$,
\item[-] $E$ is in the basin of the Dirac measure at some fixed point of $f$.
\end{itemize}
Hence, for such an $f$, if $a < R(f) / r$, then the set $\{ \chi > a \}$ is of positive Lebesgue measure but is not covered by basins of absolutely continuous measures.\\

Let us now explain our strategy to prove Theorem \autoref{thm:thm} and present the organization of the paper.
The proof relies on a reparametrization lemma, stated in section \hyperref[sec:RL]{\textbf{2.}}, which allows us to precisely control the local dynamics.
Such lemmas were introduced by Yomdin \cite{Yomdin1987VolumeGA} to prove Shub's entropy conjecture for $\mathcal{C}^{\infty}$ systems.
The reparametrization lemma that we use is an adaptation of Burguet's one from \cite{MR4701884}, which he used to prove the following result, regarding SRB measures for smooth surface diffeomorphisms:

\begin{thm*}[Burguet, \cite{MR4701884}]
Let $f: M \to M$ be a $\mathcal{C}^r$ surface diffeomorphism, where $\bbR \ni r > 1$.
There are countably many ergodic SRB measures $(\mu_i)_{i \in I}$ with $\Lambda := \{ \int \chi d \mu_i, i \in I \} \subset \left ] \frac{R(f)}{r}, + \infty \right [$, such that we have:
\begin{itemize}
\item[•] $\left \{ \chi > \frac{R(f)}{r} \right \} = \{ \chi \in \Lambda \}$ Lebesgue-almost everywhere
\item[•] For all $\lambda \in \Lambda$, we have Lebesgue-almost everywhere $\{ \chi = \lambda \} \subset \bigcup_{i, \chi(\mu_i) = \lambda} \mathcal{B}(\mu_i)$
\end{itemize}
\end{thm*}

For $\mathcal{C}^{\infty}$ surface diffeomorphisms, part of this result has also been obtained by Buzzi, Crovisier and Sarig \cite{Buzzi_Crovisier_Sarig_2022}.\\

Then, the idea is to use an entropy characterization of absolutely continuous measures. It is a version of Ledrappier-Young's entropy characterization \cite{Ledrappier_Young_1985} for endomorphisms, the precise statement for one-dimensional systems is given by Theorem \autoref{th:form-entropie} (corresponding to Theorem VII.1.1 from \cite{SETE}).
Hence, our goal is to build a hyperbolic measure $\mu$ satisfying the entropy formula $h(\mu) = \int \log |f'| \: d \mu$.
In section \hyperref[sec:mes-emp]{\textbf{5.}}, we define a measure $\mu$ as a limit of ''partial'' empirical measures: instead of averaging the Dirac measures $\delta_{f^k x}$ over all $k$ in $[\![ 0 ; n-1 ]\!]$, we only average over points of the orbit where the dynamic shows some expanding behavior, such $k$'s are called \textit{geometric times} (introduced in \cite{MR4701884} and inspired by \textit{hyperbolic times} \cite{Alves2000}).
We will use this expansion to prove that $\mu$ is hyperbolic.
We define geometric times in section \hyperref[sec:dens-geom]{\textbf{3.}}, where we also show that they happen with positive density on a set of positive Lebesgue measure --- we use the bound $R(f) / r$ to do so.
We then show in section \hyperref[sec:tpsHB]{\textbf{4.}} that geometric times are in fact hyperbolic times, in the sense of \cite{Alves2000}. A straightforward but important consequence is that if a point has an imminent geometric time, then it cannot be too close to a critical point.
In other words, our limit measure $\mu$ will avoid places where $f'$ is too close to zero, which we will use to prove that $\log |f'|$ is integrable with respect to $\mu$.\\

We explain how we estimate the entropy of $\mu$.
We first point out that our sequence of empirical measures is of the form $(\mu_n^M)_n$ where $M$ is a parameter controlling the distance to geometric times --- as $M$ grows, we allow the $k$'s in the definition of $\mu_n^M$ to get further from geometric times.
We then estimate the entropy of the empirical measures $\mu_n^M$ for some well-chosen countably infinite partition (section \hyperref[sec:entropie]{\textbf{6.}}) and show in section \hyperref[ssec:SRB]{\textbf{7.a.}} that letting $n$ and $M$ go to infinity gives entropy estimates for $\mu$.
We build this partition by dividing $I$ into monotone branches and regions where $\log |f'|$ is fixed.
The fact that $\log |f'|$ is not bounded and that $f$ may have infinitely many monotone branches explains why this partition is countably infinite.
The question of how to choose such a partition already arose when dealing with surface diffeomorphisms \cite{MR4701884}, but for one-dimensional dynamics, being able to use monotone branches makes the proof more efficient as it is very suitable with regard to the Reparametrization Lemma (see Lemma \autoref{lem:repar-monot}).
Since the chosen partition is infinite in our case, we must first show that it has finite entropy for the limit measure $\mu$ (section \hyperref[ssec:partition]{\textbf{6.b.}}).
To prove this, we use an argument from Mañé's proof of the Pesin's entropy formula \cite{Mane_Pesin}, where he shows that if the diameter of a partition is integrable with respect to some $f$-invariant measure, then this partition has finite entropy for that measure.
The fact that our partition has finite entropy will then follow from the integrability of $\log |f'|$ with respect to $\mu$.
The proof of the entropy estimate for $\mu_n^M$ is then based on a Gibbs inequality for this specific partition (section \hyperref[ssec:gibbs]{\textbf{6.c.}}).
In section \hyperref[ssec:SRB]{\textbf{7.a.}}, we put together all of the previous results to prove the entropy formula and the absolute continuity of $\mu$.
We eventually prove in section \hyperref[ssec:bassins]{\textbf{7.b.}} that the union of the basins of such measures covers $\{ x \in I \mid \chi(x) > \frac{R(f)}{r} \}$ Lebesgue-almost everywhere. We also prove the finiteness of such measures whose Lyapunov exponent is larger than some $b > R(f)/r$, and we prove the bound on the number of measures as stated in Proposition \autoref{prop:title}.

\section{Reparametrization Lemma}
\label{sec:RL}

We start by explaining the concept of a reparametrization Lemma.
The approach was introduced by Yomdin in \cite{Yomdin1987VolumeGA}, and is thus also called Yomdin's theory.
The idea is to divide the space $I$ into many small dynamically bounded pieces, all the while having an estimate of the number of pieces.
Formally, we will look for reparametrizations $\phi : [-1 ; 1] \to I$ such that the high order derivatives of $f \circ \phi$ are small. With these notations, the image of $\phi$ is one of these small pieces.
In the end, the dynamical complexity of $f$ can be understood through these reparametrizations, giving a rather combinatorial interpretation of the dynamic, which is helpful in many situations.\\

When using Yomdin theory in dimension one, many aspects are much simpler.
For example, dividing a one-dimensional space into small pieces does not need any geometric attention, while higher dimensions require using semi-algebraic geometric tools.
Another more important consequence is that it is easier to get distortion inequalities (see Lemma \autoref{lem:distortion}), which are essential to study absolutely continuous measures. We start by detailing this central fact.

\subsection{Bounded reparametrizations}
\label{ssec:reparam-bornee}


Let $r > 1$ and note $[1 ; r] = [ \! [ 1 ; \lfloor r \rfloor ] \! ] \cup \{ r \}$.
We consider the point $0$ in $I$, and we will say that a map $\sigma : [-1 ; 1] \to I$ is a \textit{reparametrization} if it is a $\mathcal{C}^r$ map whose derivative does not vanish and if we have:
$$
\sigma( \: ]-1;1] \: ) \cap \{ 0 \} = \emptyset
$$
This last condition is useful when $I$ is the circle, because it ensures the injectivity of $\sigma$.
We give more details about this in section \hyperref[ssec:diff-int-cercle]{\textbf{2.d.}}.
We will note $\sigma_* = \sigma([-1 ; 1])$ its image.

\begin{definition}[Bounded reparametrization]
A reparametrization $\sigma$ is said to be bounded if
$$
\max\limits_{s \in ]1 ; r]}|| d^s \sigma ||_{\infty} \leq \frac{1}{6} || \sigma' ||_{\infty}
$$
Then, for $\varepsilon > 0$, it is said to be $\varepsilon$-bounded if it also satisfies
$$
|| \sigma' ||_{\infty} \le \varepsilon
$$
Moreover, we say it is $(n , \varepsilon)$-bounded for $f : I \to I$ if we have
$$
\forall k \in [ \! [ 0 ; n ] \! ], f^k \circ \sigma \text{ is } \varepsilon \text{-bounded}
$$
\end{definition}

As mentioned before, the important property satisfied by bounded reparametrizations is the following control of the distorsion:

\begin{lem}[Distortion inequality]
\label{lem:distortion}
If $\sigma$ is a bounded reparametrization, then we have
$$
\forall t,s \in [-1 ; 1], \frac{| \sigma'(t) |}{|\sigma'(s)|} \leq \frac{3}{2}
$$
\end{lem}

\begin{proof}
Consider $s_0 \in [-1;1]$ such that $| \sigma'(s_0) | = ||\sigma'||_{\infty}$. By noting $r' = \min(2,r)$, we get
$$
|\sigma'(s) - \sigma'(s_0)| \leq |s-s_0|^{r'-1} || d^{r'} \sigma ||_{\infty}
\leq 2 \frac{1}{6} |\sigma'(s_0)|
$$
This leads to
$$
|\sigma'(s)| \geq |\sigma'(s_0)| -  |\sigma'(s_0) - \sigma'(s)| \geq \frac{2}{3} | \sigma'(s_0)|
$$
which does give
$$
\frac{| \sigma'(t) |}{|\sigma'(s)|} \leq \frac{| \sigma'(s_0) |}{|\sigma'(s)|} \leq \frac{3}{2}
$$
\end{proof}

We now state a lemma about bounded reparametrizations that we will use in the next section.

\begin{lem}[Lemma 6 from \cite{MR4701884}]
\label{lem:decoup-arbre-geom}
Let $\gamma : [-1 ; 1] \to I$ be a bounded reparametrization satisfying $|| \gamma' ||_{\infty} \geq \varepsilon$. 
Then there exists a finite family of affine maps $(\iota_j : [-1 ; 1] \to [-1 ; 1])_{j \in L}$ where $L$ is of the form $\overline{L} \sqcup \underline{L}$ and such that
\begin{itemize}
\item[i)] For any $j \in L$, the map $\gamma \circ \iota_j$ is an $\varepsilon$-bounded reparametrization and $|| (\gamma \circ \iota_j)'(0) || \geq \varepsilon / 6$
\item[ii)] $[-1 ; 1] = \left ( \bigcup\limits_{j \in \underline{L}} \iota_j([-1 ; 1]) \right ) \cup \left ( \bigcup\limits_{j \in \overline{L}} \iota_j\left ( \left [ - \frac{1}{3} ; \frac{1}{3} \right ] \right ) \right )$
\item[iii)] $\# \underline{L} \leq 2$ and $\# \overline{L} \leq 6 \left ( \frac{|| \gamma' ||_{\infty}}{\varepsilon} + 1\right )$
\item[iv)] For $x \in \gamma_*$, we have $\# \{ j \in L \mid (\gamma \circ \iota_j)_* \cap B(x, \varepsilon) \neq \emptyset \} \leq 100$
\end{itemize}
\end{lem}


\subsection{Reparametrization Lemma}

We first introduce some more notations. We will note 
$\gls*{Crit} = Crit(f) = \{ x \in I \mid f'(x) = 0 \}$, and for any 
$n \in \gls*{Z_plus}=[\![ 1 ; + \infty [\![$, we let
$$
\gls*{Critn} = Crit(f^n) \; \; \; \text{and} \; \; \; \mathcal{C}_{\infty} = \bigcup\limits_{n \in \bbZ^+} \mathcal{C}_n
$$
For $z \in I \backslash \gls*{Critinf}$ and $n \in \bbZ^+$, note
\begin{align*}
\gls*{Kng} &= \lfloor \log^+ |g'(g^{n-1}(z))| \rfloor \in [\![ 0 ; \log || g' ||_{\infty} ]\!]\\
\gls*{Kng_prime} &= \lfloor \log^- |g'(g^{n-1}(z))| \rfloor \geq 0
\end{align*}
where $\gls*{part_entiere}$ is the largest integer smaller or equal to $x$.
Lastly, for 
$k,k' \in (\gls*{Zz_plus})^n$, where $\bbZ_0^+ = \bbZ^+ \cup \{ 0 \}$,we let
$$
\gls*{Hg} = \left \{ x \in I \mid k_{i,g}(x) = k_i \text{ and } k_{i,g}'(x) = k_i' \text{ for any } i \in [\![ 1 ; n ]\!] \: \right \}
$$
Let us explain the purpose of this notation.
Yomdin's original Reparametrization Lemma was local in space, meaning that the part of the space that would get divided into many pieces is not $\sigma_*$, but only the intersection of $\sigma_*$ with a small enough dynamical ball.
However, similarly to Burguet's Reparametrization Lemma from \cite{MR4701884}, our Reparametrization Lemma will be global in space, but local in terms of values of the $\log$ of the derivative.
More precisely, we will not reparametrize $\sigma_*$ completely, but only the intersection of $\sigma_*$ with $\mathcal{H}_g(k,k')$, for any $k, k'$.\\

We now give a first taste of the forthcoming reparametrization lemma.
We start by considering a reparametrization $\sigma$, and we assume that its image $\sigma_*$ is small enough --- the precise bound on this size is given in the proof of the \hyperref[lem:RL]{Reparametrization Lemma}, it depends on $f$ and $r$ and it allows us to approximate $f$ by its Taylor polynomial.
Fix $k_1,k_1' \in \bbZ_0^+$, and divide the intersection $\sigma_* \cap \mathcal{H}_g(k_1,k_1')$ into many pieces so that $f$ composed with the reparametrizations corresponding to these pieces has bounded distorsion. This gives a first family of reparametrizations.
Then, to control the distorsion of $f^n$, we iterate the construction that allowed us to bound the distorsion of $f^1$.
In particular, if $\theta$ is one of the reparametrizations in the above family, then the set $(f \circ \sigma \circ \theta)_*$ must be small enough, which is not necessarily the case.
To ensure that it is small enough, we apply Lemma \autoref{lem:decoup-arbre-geom} to the reparametrization $f \circ \sigma \circ \theta$.
In other words, for every $\theta$ such that the image $(f \circ \sigma \circ \theta)_*$ is not small enough, we divide the set $(\sigma \circ \theta)_*$ into even more pieces.
In the end, we obtain reparametrizations that allow us to bound the distorsion of $f^1$ and whose images by $f$ are small enough.
We denote this family of reparametrizations by $\mathcal{T}_1^{(k_1,k_1')}$.
Then, to control the distortion of $f^n$, we proceed by induction and follow the steps shown in Figure 1:

\vspace{1em}
%


\begin{figure}[h!]
\centering
\includegraphics[width=37em]{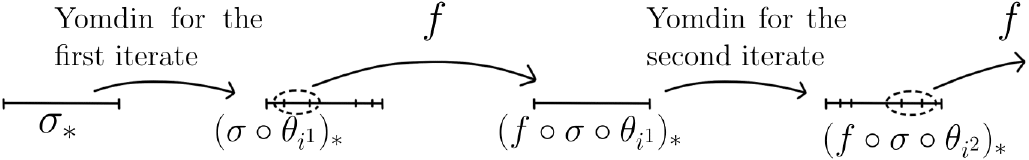}
\vspace{0.5em}
\caption{Iterating Yomdin's division process to bound the distortion of iterates of $f$}
\vspace{1.5em}
\end{figure}

For this induction, formally, we consider a reparametrization $\theta \in \mathcal{T}_1^{(k_1,k_1')}$, and we reparametrize the set $(f \circ \sigma \circ \theta)_* \cap \mathcal{H}_g(k_2,k_2')$ with fixed $k_2, k_2' \in \bbZ_0^+$, so that $f$ has bounded distortion through these new reparametrizations.
Hence if we note $(\phi_{\theta}^i)_i$ these new reparametrizations, then the reparametrizations $\theta \circ \phi_{\theta}^i$ cover $\sigma^{-1} \mathcal{H}_g((k_1,k_1'),(k_2,k_2'))$ and bound the distortion of $f^2$.
By doing it for every $\theta \in \mathcal{T}_1^{(k_1,k_1')}$, we get a second level of reparametrizations nested in $\mathcal{T}_1^{(k_1,k_1')}$ where $f^2$ has bounded distortion, call it $\mathcal{T}_2^{((k_1,k_1'),(k_2,k_2'))}$.\\

So if we fix $k=(k_1, ..., k_n)$ and $k'=(k_1', ..., k_n')$ and note $\mathcal{T}_n^{(k,k')}$ the $n$-th level of reparametrizations of $\mathcal{H}_g(k,k')$, we get a tree structure as represented in Figure 2.
As for notations, we write $(\gls*{theta_n})_{i^n}$ the $n$-th level of reparametrizations and $\gls*{theta_nk}$ the child of $\theta_{i^n}$ at level $k$.
\vspace{1em}

\begin{figure}[h!]
\centering
\includegraphics[height=24em]{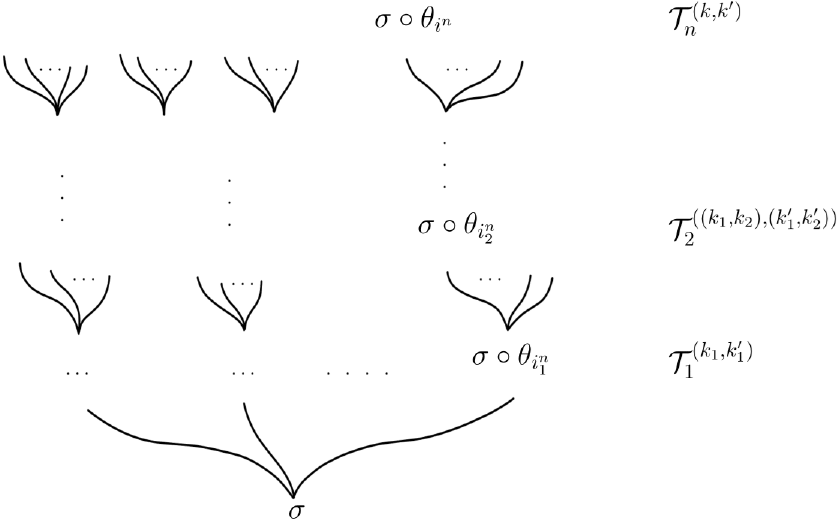}
\vspace{0.5em}
\caption{A tree where children of a node $\sigma \circ \theta$ are of the form $\sigma \circ \theta \circ \varphi$ where $\varphi$ is an affine contraction}
\vspace{1.5em}
\end{figure}

We build these branches for each $\mathcal{H}_g(k,k')$, and label every vertex in $\mathcal{T}_n^{(k,k')}$ with $(k_n, k_n')$ to know from which $\mathcal{H}_g(k,k')$ it comes from.
Then, we note $\gls*{tree_n}$ the union of all $\mathcal{T}_n^{(k,k')}$ for $k,k' \in (\bbZ_0^+)^n$.
Notice that this tree may have unbounded degree, although its restriction to the vertices labeled by a given $(k,k')$ has bounded degree (see item $5)$ of Lemma \autoref{lem:RL}).\\

Lastly, recall that we had to use Lemma \autoref{lem:decoup-arbre-geom} in the construction.
The purpose was to ensure that the image of a reparametrization was small enough to approximate $f$ by its Taylor polynomial.
For example, in the first step of the induction, we reparametrized the $\theta$'s for which $( f \circ \sigma \circ \theta)_*$ was not small enough.
When we have to make this additional division, we will consider that the resulting pieces are the ones that see the expansion of the dynamic.
To keep track of these pieces, we will say that a vertex $i^n$ of the tree $\mathcal{T}_n$ can be of two different types: write $i^n \in \gls*{tree_n_geom}$ for the ones that see some expansion and $i^n \in \underline{\mathcal{T}_n}$ for the others.
Then, by replacing $f$ with some iterate $f^p = g$, we get this Reparametrization Lemma:

\begin{lem}[Reparametrization Lemma]
\label{lem:RL}
For $p \in \bbZ^+$, there exists $\varepsilon > 0$ such that, for any $\varepsilon$-bounded reparametrization $\sigma : [-1 ; 1] \to I$, we have a tree $\mathcal{T}$ and affine maps $(\theta_{i^n} : [-1 ; 1] \to [-1 ; 1])_{i^n \in \mathcal{T}_n}$ such that for any $n$:\\

\hspace{1em}1) For any $i^n \in \mathcal{T}_n$, the reparametrization $\sigma \circ \theta_{i^n}$ is $(n, \varepsilon)$-bounded for $g=f^p$\\

\hspace{1em}2) For any $i^n \in \mathcal{T}_n$, the affine map $\theta_{i^n}$ is of the form $\theta_{i^n_{n-1}} \circ \varphi_{i^n}$ where $\varphi_{i^n}$ is an affine contraction of rate smaller than $1/100$, and when $i_{n-1}^n \in \overline{\mathcal{T}_{n-1}}$, we have $\theta_{i^n}([-1 ; 1]) \subset \theta_{i_{n-1}^n}([- 1/3 ; 1/3])$\\

\hspace{1em}3) For any $i^n \in \overline{\mathcal{T}_n}$, we have $| (g^n \circ \sigma \circ \theta_{i^n})' (0) | \geq \varepsilon / 6$\\

\hspace{1em}4) For any $k^n, k'^n \in (\bbZ_0^+)^n$, the set $\sigma^{-1} \mathcal{H}_g(k^n, k'^n)$ is contained in the following union
$$\left(\bigcup\limits_{
\substack{
i^n \in \overline{\mathcal{T}_n} \text{ such that}\\
k^{(')}(i^n) = {k^{(') n}}}}
\theta_{i^n}([-1/3 ; 1/3]) \right)
\bigcup
\left ( \bigcup\limits_{\substack{
i^n \in \underline{\mathcal{T}_n} \text{ such that}\\
k^{(')}(i^n) = {k^{(') n}}}}
\theta_{i^n}([-1 ; 1]) \right )
$$
each set of these unions having non empty intersection with $\sigma^{-1} \mathcal{H}_g(k^n, k'^n)$\\

\hspace{1em}5) For $i^{n-1} \in \mathcal{T}_{n-1}$ and $k_n' \in \bbZ_0^+$, we have
$$
\# \{ i^n \in \overline{\mathcal{T}_n} \mid i^{n-1} = i^n_{n-1}\text{ and } k_n'(i^n) = k_n' \} \leq C_r \log || g' ||_{\infty} e^{\max \left ( \log || g' ||_{\infty}, \frac{k_n'}{r-1} \right )}
$$
$$
\# \{ i^n \in \underline{\mathcal{T}_n} \mid i^{n-1} = i^n_{n-1}\text{ and } k_n'(i^n) = k_n' \} \leq C_r \log || g' ||_{\infty} e^{\frac{k_n'}{r-1}}
$$\\[-1em]

\hspace{1em}6) $I \backslash \mathcal{C}_n = \bigcup\limits_{k^n, k'^n \in (\bbZ_0^+)^n} \mathcal{H}_g(k^n, k'^n)$ and $\bigcup\limits_{i^n \in \mathcal{T}_n} (\sigma \circ \theta_{i^n})_* \supset \sigma([-1 ; 1]) \backslash \mathcal{C}_n$
\end{lem}

\begin{proof}
Note $r' = \min(2,r)$, and fix $\varepsilon > 0$ satisfying
$$
(2 \varepsilon)^{r'-1} < \frac{1}{2 || g' ||_{r - 1}}
$$
For $x \in I$, note $g_{2 \varepsilon}^x : t \in [-1 ; 1] \mapsto g(x + 2 \varepsilon t) \in I$.
We first show the following
$$
\forall x \in I, \forall s \in \: [1 ; r], || d^s \left ( g_{2 \varepsilon}^x \right ) ||_{\infty} \leq 3 \varepsilon \max \left ( 1, | g'(x)| \right )
$$
When $s \in ]1 ; r]$, we have $s \geq r'$, so
\begin{align*}
| d^s \left ( g_{2 \varepsilon}^x \right ) (t) | &= (2 \varepsilon)^s | d^s g (x + 2 \varepsilon t)|\\
&\leq (2 \varepsilon)^{r'} || g' ||_{r-1}\\
&\leq 2\varepsilon \times 1/2 
\end{align*}
And if $s=1$, then
\begin{align*}
| \left ( g_{2 \varepsilon}^x \right )' (t) | &= 2 \varepsilon |g'(x + 2 \varepsilon t)|\\
&\leq 2 \varepsilon \left ( |g'(x + 2 \varepsilon t) - g'(x)| + |g'(x)| \right )\\
&\leq 2 \varepsilon | 2 \varepsilon t |^{r'-1} || d^{r'} g||_{\infty} + 2 \varepsilon |g'(x)|\\
&\leq 3 \varepsilon \max (1, |g'(x)|)
\end{align*}

Once this $\varepsilon$ is defined, the proof is the same as for the Reparametrization Lemma in \cite{MR4701884}, except for two differences.
The first one relates to item $5)$, where we count the number of vertices of $\overline{\mathcal{T}_n}$ and $\underline{\mathcal{T}_n}$ for which $k_n'$ takes a given value.
However in \cite{MR4701884}, both the values of $k_n'$ and $k_n$ were fixed.
Hence, to obtain item $5)$ as stated here, we simply multiply the bound in \cite{MR4701884} by the number of all possible $k_n$, which explains the term $\log || g' ||_{\infty}$ in our bound.
The second difference is related to our definition of reparametrizations.
More precisely, we have to show that the maps $g^k \circ \sigma \circ \theta_{i^n}$ are reparametrizations, for $k \leq n$, in the sense that $\left ( g^k \circ\sigma \circ \theta_{i^n} \right )(]-1 ; 1]) \cap \{ 0 \} = \emptyset$.
To guarantee that this condition is satisfied, we make an additional division before applying Lemma \autoref{lem:decoup-arbre-geom}. It was not needed in \cite{MR4701884} and it multiplies the constant $C_r$ in item $5)$ by $2$.

\end{proof}

\subsection{The case of circle maps}
\label{ssec:diff-int-cercle}

A first technical remark is that when computing derivatives of circle maps, one should take care of the tangent bundle.
However, we do all the computations as for interval maps, which does not change the arguments.\\

A more important difference relates to monotonicity.
For interval maps, if we have some subinterval where $f'$ does not vanish, then $f$ is injective on this subinterval, but this is not true for circle maps.
With this in mind, consider the point $0$ in $I$, and define monotone branches as follows:

\begin{definition}[Monotone branch]
\label{def:monot-branche}
The monotone branches of $f$ are the sets of the form $[a;b[$ where $]a ;b[$ is a connected component of the following set
$$
\{ x \in I \mid f'(x) \neq 0 \text{ and } f(x) \neq 0 \}
$$
\end{definition}

The main consequence of this definition is Lemma \autoref{lem:repar-monot}: if $\sigma$ is an $(n, \varepsilon)$-bounded reparametrization for $f$, then it is contained in the closure of a monotone branch of $f^n$.
Furthermore, the following lemma implies that this definition of monotone branches does not change anything for interval maps.

\begin{lem}
\label{lem:preim-bords}
If $I = [ 0 ; 1]$, then
$$
f^{-1}(\{ 0 ; 1 \}) \subset \{ 0 ; 1 \} \cup \mathcal{C}
$$
\end{lem}

\section{Positive density of geometric times}
\label{sec:dens-geom}

We start by defining geometric times.
Let $p \in \bbZ^+$ and let $g = f^p$. 
In Proposition \autoref{prop:dens-pos}, we will choose $p$ large.
Let $\varepsilon > 0$ be such that the \hyperref[lem:RL]{Reparametrization Lemma} applies to $g$. Let $\sigma$ be an $\varepsilon$-bounded reparametrization and $x \in \sigma_*$.
Note
$$
\gls*{Ep} := \left \{ m \in \bbZ_0^+ \mid \exists i^m \in \overline{\mathcal{T}_m}, k'(i^m) = k'^m(x), x \in \sigma \circ \theta_{i^m} \left ( \left [ - \frac{1}{3} ; \frac{1}{3} \right ] \right ) \right \}
$$
As in \cite{MR4701884}, elements of this set will be called \textit{geometric times}.
The idea is that if $n$ is a geometric time for $x$, then there is some small neighborhood of $x$ such that:
\begin{itemize}
\item[i)] For the first $n$ iterates of $g$, we have bounded distortion on this neighborhood
\item[ii)] The size of this neighborhood decreases exponentially in $n$, while the size of its image by $g^n$ is uniformly bounded from below
\end{itemize}

We now define several notions of density. First, we define the upper density of a subset $E$ of $\bbZ_0^+$ by
$$
\gls*{D_upp} = \limsup\limits_{n \to +\infty} \frac{1}{n} \# E \cap [\![ 0 ; n [\![
$$
Before taking the limit, we note for any $n \in \mathbb{Z^+}$
$$
\gls*{D_part} = \frac{1}{n} \# E \cap [\![ 0 ; n [\![ 
$$
We also define the density along a specific integer sequence $\mathcal{n}$ by
$$
\gls*{D_along} = \lim\limits_{\mathcal{n} \ni n \to +\infty} d_n(E)
$$

We now explain how the bound $R(f) / r$ in Theorem \autoref{thm:thm} will force positive upper density of $E_p(x)$.
It is related to the growth rate of the number of pieces we get when we follow the reparametrization scheme.
Let us recall that there are two types of reparametrizations, as explained in the last paragraph before stating the \hyperref[lem:RL]{Reparametrization Lemma}:
\begin{itemize}
\item[-] The ones in $\underline{\mathcal{T}_n}$ that directly gave $(n, \varepsilon)$-bounded reparametrizations for $g = f^p$. Their number has an exponential growth rate given by $k' / r-1$ (where $k'$ is some $k'_{n,g}(x)$).
\item[-] The ones in $\overline{\mathcal{T}_n}$ that required an additional division to be $(n,\varepsilon)$-bounded. These are the ones that see the expansion of the system, and their exponential growth rate is given by $\max (k' / r-1, \log ||g'||_{\infty} )$.
\end{itemize}
So, if the number of pieces has an exponential growth rate larger than $k' / r-1$, then the second type has to appear often enough. This is achieved if the system is expanding enough, and we will see that a Lyapunov exponent larger than $R(f) / r$ on a set of positive Lebesgue measure exactly corresponds to this large enough expansion.

\begin{prop}
\label{prop:dens-pos}
For $b > R(f) / r$, there exists $p_0$ such that the following property holds.
For any $p \geq p_0$, there exist $\beta = \beta_p > 0$ and $\varepsilon = \varepsilon_p > 0$ such that, if $\sigma$ is an $\varepsilon$-bounded reparametrization, then
$$
\limsup\limits_{n \to +\infty} \frac{1}{n} \log Leb \left ( \left \{ x \in \sigma_* \mid d_{\lfloor n / p \rfloor}(E_p(x)) < \beta \text{ and } | (f^n)'(x) | \geq e^{nb} \right \} \right ) < 0
$$
\end{prop}

\begin{proof}
Let $p \in \bbZ^+$. Let $\varepsilon := \varepsilon_p > 0$ be given by the \hyperref[lem:RL]{Reparametrization Lemma} for $g = f^p$.
Let $\beta > 0$ be a quantity that we will determine at the end of the proof.
Let $\sigma$ be an $\varepsilon$-bounded reparametrization.
We start with the case where $n = mp$.
Consider the tree given by the \hyperref[lem:RL]{Reparametrization Lemma} up to level $m$.
Take $x \in \sigma_*$ such that
$$
d_m(E_p(x)) < \beta \text{ and } |(f^n)'(x)| \geq e^{nb}
$$
and note $(k_1', ..., k_m') = (k_1'(x), ..., k_m'(x))$. Then, by item $4)$ of the \hyperref[lem:RL]{Reparametrization Lemma}, we have some $i^m \in \mathcal{T}_m$ such that
$$
k'^m(i^m) = (k_1', ..., k_m') \; \; \; \text{and} \; \; \; x \in \left ( \sigma \circ \theta_{i^m} \right )_*
$$
These $k_i'$ satisfy the following inequality
\begin{align*}
\sum\limits_{i=1}^m k_i'(i^m)
&= \sum\limits_{i=1}^m k_i'(x)\\
&\leq \sum\limits_{i=0}^{m-1} \lfloor \log || g' ||_{\infty} \rfloor - \lfloor \log |g'(g^i x)| \rfloor\\
&\leq m \log || (f^p)' ||_{\infty} + m - \log |(f^n)'(x)|\\
&\leq m \left ( \log || (f^p)'||_{\infty} + 1 - pb \right)
\end{align*}
Now, for $S \in \bbR_+$, we have the following general statement
\begin{align}
\# \left \{ (k_1', ..., k_m') \in \bbZ_0^+ \mid \sum\limits_{i = 1}^m k_i' \leq m S \right \} \leq {{m(S+1)}\choose m} \leq e^{m \log (S+1) + m}
\label{ineq:3-1}
\end{align}
We will hence choose $S = \log || (f^p)'||_{\infty} + 1 - pb$. Then, consider the following inequality
\begin{align*}
Leb&\left ( \left \{ x \in \sigma_* \mid d_m(E_p(x)) < \beta \text{ and } | (f^n)'(x) | \geq e^{nb}\right \} \right )\\
&\leq \sum\limits_{\substack{(k_i')_{i \in [\![ 1 ; m ]\!]} \text{ t.q.}\\
\sum k_i' \leq m S}} \sum\limits_{\substack{i^m \in \mathcal{T}_m \text{ t.q.}\\
k'^m(i^m) = (k_i')\\
\text{et } \# \{ k \in [\![ 1 ; m ]\!] \mid i^m_k \in \overline{\mathcal{T}_k} \} < m \beta}}
\hspace{-2em}Leb ((\sigma \circ \theta_{i^m})_* \cap \{ |(f^n)'| \geq e^{nb} \})
\end{align*}
Since $g^m$ is injective on the image of the $(n, \varepsilon)$-bounded reparametrization $\sigma \circ \theta_{i^m}$, the change of variable formula (Lemma \autoref{lem:chgt-var}) gives
$$
Leb ((\sigma \circ \theta_{i^m})_* \cap \{ |(f^n)'| \geq e^{nb} \}) \leq \left ( \inf\limits_{(\sigma \circ \theta_{i^m})_* \cap \{ |(f^n)'| \geq e^{nb} \}} |(g^m)'|\right )^{-1} Leb ((g^m \circ \sigma \circ \theta_{i^m})_*)
$$
Furthermore, the reparametrization $g^m \circ \sigma \circ \theta_{i^m}$ is $\varepsilon$-bounded, so
$$
Leb ((\sigma \circ \theta_{i^m})_* \cap \{ |(f^n)'| \geq e^{nb} \}) \leq 2 \varepsilon e^{-nb}
$$
Then, for each $i^m \in \mathcal{T}_m$ and $k \in [\![ 1 ; m ]\!]$, we consider the vertices $i^m_k$ in $\mathcal{T}_k$, i.e. the parents of $i^m$, and we note $j$ the number of $k \in [\![ 1; m ]\!]$ such that $i_k^m \in \overline{\mathcal{T}_k}$. Hence, item $5)$ of the \hyperref[lem:RL]{Reparametrization Lemma} gives
\begin{align*}
Leb&\left ( \left \{ x \in \sigma_* \mid d_m(E_p(x)) < \beta \text{ and } | (f^n)'(x) | \geq e^{nb}\right \} \right )\\
&\leq 2 \varepsilon e^{- nb} \sum\limits_{\substack{(k_i')_{i \in [\![ 1 ; m ]\!]} \text{ t.q.}\\
\sum k_i' \leq m S}} \hspace{1em} \sum\limits_{j = 0}^{m \beta} \; \; \sum\limits_{1 \leq i_1 < ... < i_j \leq m} \left ( \prod\limits_{s=1}^j C_r \log || g' ||_{\infty} e^{\max \left ( \log || g' ||_{\infty}, \frac{k_{i_s}'}{r-1} \right )} \right )\\
&\hspace{17em}\times \left ( \prod\limits_{s \in [\![ 1 ; m ]\!] \backslash \{ i_1, ..., i_j \}} C_r \log || g' ||_{\infty} e^{\frac{k_s'}{r-1}} \right )\\
\text{using } (3.1) \hspace{2em} &\leq 2 \varepsilon e^{-nb} \times e^{m \log (S + 1) + m} \times 2^m \times C_r^m \left (  \log || g' ||_{\infty} \right )^m e^{\frac{m S}{r-1}} \times || f' ||_{\infty}^{p \times \beta m}
\end{align*}
Notice that $S/p \underset{p \to +\infty}{\longrightarrow} R(f) - b$, therefore
\begin{align*}
\limsup\limits_{p \bbZ^+ \ni n \to +\infty} \frac{1}{n} \log Leb_{\sigma_*} &\left ( \left \{ x \in \sigma_* \mid d_n(E_p(x)) < \beta \text{ and } | (f^n)'(x) | \geq e^{nb} \right \} \right )\\
&\leq -b + \frac{\log(S+1) + 1}{p} + \frac{\log (2 C_r p \log || f ' ||_{\infty})}{p}
+ \frac{S}{p(r-1)} + \beta \log || f' ||_{\infty}\\
&= -b + \frac{1}{r-1}(R(f) - b) + \beta \log || f' ||_{\infty} + \underset{p \to +\infty}{o}(1)\\
&=\frac{r}{r-1}\left ( \frac{R(f)}{r} - b \right ) + \beta \log || f' ||_{\infty} + \underset{p \to +\infty}{o}(1)
\end{align*}
Since $\frac{R(f)}{r} - b < 0$, this concludes the case $n=mp$.\\

For the general case, write $n = mp + s$ with $s \in [\![ 0 ; p-1 ]\!]$. We note $n' = mp$ and show that the previous argument still works.
Let $x \in \sigma_*$ be such that
$$
d_m(E_p(x)) < \beta \text{ and } |(f^n)'(x)| \geq e^{nb}
$$
Then
$$
|(f^{n'})'(x)| \: = \: |(f^{-s + n})'(x)|
\: \geq \: \frac{| (f^n)'(x)|}{|| f' ||_{\infty}^s}
\: \geq \: \frac{1}{|| f' ||_{\infty}^p} e^{nb}
\geq \: \frac{1}{|| f' ||_{\infty}^p} e^{n'b}
$$
Previous computations for $n'$ lead to an upper bound on the Lebesgue measure for $n$ and we would reach the same conclusion.
\end{proof}

\section{Geometric times are hyperbolic times}
\label{sec:tpsHB}

In this section, we show some useful properties of geometric times that emphasize the relationship between geometric and hyperbolic times, as defined by Alves, Bonatti and Viana \cite{Alves2000}.\\

We first introduce some notations that we will use throughout the rest of the paper.\\

For $E \subset \bbZ_0^+, M \in \bbZ_0^+$ and $n \in \bbZ_0^+$, we let
$$
\gls*{EnM} = \bigcup\limits_{\substack{n > k,l \in E\\
|k-l| \leq M}} [ \! [ k ; l [ \! [
$$
Then, for $m \in \bbZ_0^+$, we define the set $\gls*{EnMm}$ as follows:
for each connected component of $E_n^M$, we remove its last $L$ elements, where $L$ depends on the connected component and is the minimal integer such that $m-1 \leq L \leq M+m-2$ and such that the new connected component is still of the form $[\![ k ; l [\![$ with $k,l \in E$.
In particular, notice that $E_n^{M,1} = E_n^M$.
Notice as well that for $n$ and $m$ fixed, the set $E_n^{M,m}$ is non-decreasing in $M$.
Indeed, consider the set $\widetilde{E}_n^{M,m}$ where we only removed the last $m-1$ elements of each connected component of $E_n^M$.
Hence, for $M \leq M'$, the inclusion $E_n^M \subset E_n^{M'}$ gives $\widetilde{E}_n^{M,m} \subset \widetilde{E}_n^{M',m}$. Then, to build $E_n^{M,m}$, we removed more and more elements at the end of each connected component of $\widetilde{E}_n^{M,m}$ until we reached an element of $E$.
Since $\widetilde{E}_n^{M,m} \subset \widetilde{E}_n^{M',m}$, the element of $E$ that we reach in $\widetilde{E}_n^{M',m}$ cannot come after the one in $\widetilde{E}_n^{M,m}$, which gives $E_n^{M,m} \subset E_n^{M',m}$.
We will also note $\gls*{E_front} = E \Delta (E+1)$ where $\Delta$ denotes the symmetric difference.
We now prove some general properties of these sets.

\begin{lem}
\label{lem:EnM}
For $M,n,m \in \bbZ_0^+$, the set $E_n^{M,m}$ satisfies the following properties:
\begin{itemize}
\item[i)] $\partial E_n^{M,m} \subset E$
\item[ii)] $\limsup\limits_{n \to + \infty} d_n(\partial E_n^{M,m}) \leq \frac{2}{M}$
\item[iii)] $M \frac{\# \partial E_n^{M,m}}{2} \leq n + M$
\item[iv)] For $M' \geq M$, $\# (E_n^{M',m} \backslash E_n^{M,m})
\geq M \frac{\# \partial E_n^{M,m} - \# \partial E_n^{M',m}}{2}$
\end{itemize}
\end{lem}

\begin{proof}
To prove $i)$, notice that the set $E_n^{M,m}$ is a union of intervals $[\![ k : l [\![$ whose boundary $\{ k ; l \}$ lies in $E$ and that for any $A,B \subset \bbZ_0^+$, we have $\partial(A \cup B) \subset \partial A \cup \partial B$.\\

Then, $ii)$ follows from $iii)$, so we only prove $iii)$ and $iv)$. To do so, we consider the following figure:
\vspace{1em}

\begin{figure}[h!]
\centering
\includegraphics[width=32em]{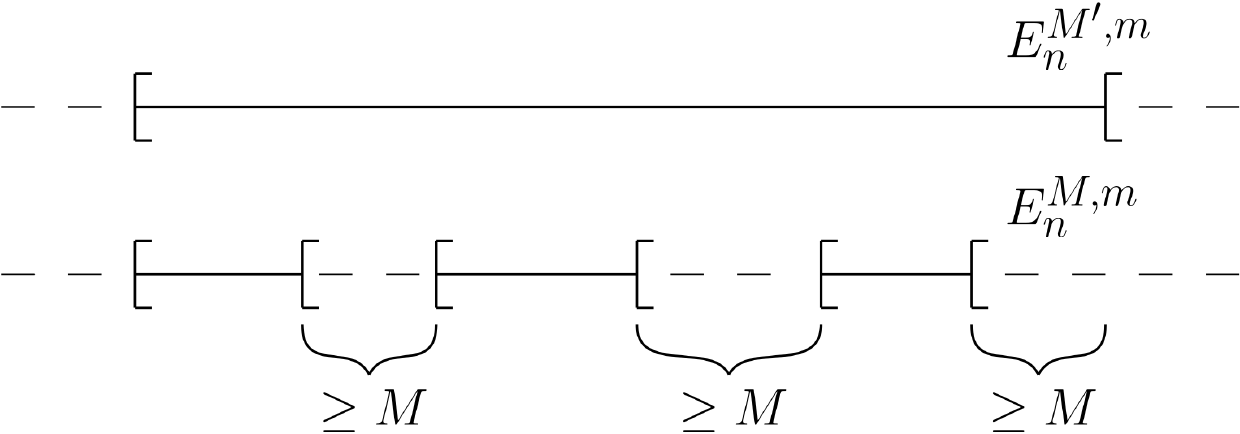}
\vspace{0.5em}
\caption{A visualization of the sets $E_n^{M,m}$ and $E_n^{M',m}$}
\vspace{1.5em}
\end{figure}

- Proof of $iii)$: For each connected component $[\![ k ; l [\![$ of $E_n^{M,m}$, consider the interval $[\![ l ; l+M [\![$ contained in $[\![ 0 ; n + M [\![$. This gives $\frac{\# \partial E_n^{M,m}}{2}$ disjoint intervals of length $M$ inside $[\![ 0 ; n + M [\![$.\\

- Proof of $iv)$: For each connected component $[\![ k ; l [\![$ of $E_n^{M,m}$ such that $l \notin \partial E_n^{M',m}$, consider the set $[\![ l ; l+M [\![$ contained in $E_n^{M',m} \backslash E_n^{M,m}$.
This gives at least $\frac{\# \partial E_n^{M,m} - \# \partial E_n^{M',m}}{2}$ disjoint intervals of length $M$ inside $E_n^{M',m} \backslash E_n^{M,m}$.
\end{proof}

In what follows, $E$ will be the set of geometric times of some point.
Let $p \in \bbZ^+$. Let $\varepsilon := \varepsilon_p > 0$ be given by the \hyperref[lem:RL]{Reparametrization Lemma} for $g = f^p$.
Let $\sigma$ be an $\varepsilon$-bounded reparametrization.
For $x \in \sigma_*$, $n, M,m \in \bbZ_0^+$, we note
$\gls*{EnMmx} = (E_p(x))_n^{M,m}$, and we let $E^{M,m}(x)$ be the non-decreasing limit in $n$ of $E_n^{M,m}(x)$. Notice that we may have $E^{M,m}(x) \cap [\![ 0 ; n [\![ \: \neq E_n^{M,m}(x)$, although these sets are equal up to the last $2M+m$ elements.\\

In the next statement, the term ''$\log 10$'' will only be of use to prove that the Lyapunov exponent of $g$ is strictly positive $\mu$-almost everywhere (Proposition \autoref{prop:expo-pos}).

\begin{lem}
\label{lem:tpsHB}
For $x \in \sigma_*, n, M,m \in \bbZ_0^+$, we have the following properties:
\begin{itemize}
\item[i)] For $k < l$ and $l \in E(x)$, we have
$$
\log |(g^{l-k})'(g^k x)| \geq (l-k) \log 10 \geq 0
$$
\item[ii)] If $[\![ a ; b [\![$ is a connected component of $E_n^{M,m}(x)$, then
$$
\log |(g^{b-a})'(g^a x)| \geq (b-a) \log 10 \geq 0
$$
\item[iii)] For $[\![ k ; l [\![ \subset E_n^{M,m}(x)$, we have
$$
\log |(g^{l-k})'(g^k x)| \geq (l-k) \log 10 - M \log || g' ||_{\infty} \geq - M \log || g' ||_{\infty}
$$
\end{itemize}
\end{lem}

\begin{proof}
\hspace{1em}- Proof of $i)$:
By definition of $E(x)$, we may write $x = \sigma \circ \theta_{i^l}(t)$ with $\theta_{i^l} = \theta_{i^k} \circ \varphi_{i_k^l}$ and $i^l \in \overline{\mathcal{T}_l}$ (i.e. the branch of $\mathcal{T}_l$ that ends at $i^l$ passes through some $i^k$, and we note $\varphi_{i_k^l}$ the contraction that goes from $i^k$ to $i^l$).
We hence have
\begin{align*}
|(g^{(l-k)})'(g^k x) | = \frac{| (g^l)'(x) |}{| (g^k)'(x) |} &= \frac{| (g^l \circ \sigma \circ \theta_{i^l})'(t) |}{| (g^k \circ \sigma \circ \theta_{i^k})'(\varphi_{i_k^l}(t)) | \times | (\varphi_{i_k^l})'(t) |}\\
\text{1) and 2) from the \hyperref[lem:RL]{Reparametrization Lemma}\hspace{2em}}&\geq \frac{2}{3} \frac{| (g^l \circ \sigma \circ \theta_{i^l})'(0) |}{\varepsilon} 100^{l-k}\\
\text{3) from the \hyperref[lem:RL]{Reparametrization Lemma}\hspace{2em}}&\geq \frac{2}{3} \frac{1}{6} 100^{l-k}\\
&\geq 10^{l-k}
\end{align*}

\hspace{1em}-  Proof of $ii)$:
From item $i$) from Lemma \autoref{lem:EnM}, we have that $\partial E_n^{M,m}(x) \subset E(x)$, so we apply $i)$.\\

\hspace{1em}- Proof of $iii)$:
Let $b$ be the first element of $E(x)$ that is superior or equal to $l$.
By definition of $E_n^{M,m}(x)$, we have $|b-l| < M$.
Therefore, we get
\begin{align*}
\log |(g^{l-k})'(g^k x)|
&= \left ( \sum\limits_{i \in [\![ k ; b [\![} - \sum\limits_{i \in [\![ l ; b [\![} \right ) \log |g'(g^i x)|\\
\underset{i)}&{\geq} (b-k) \log 10 - M \log || g' ||_{\infty}\\
&\geq (l-k) \log 10 - M \log || g' ||_{\infty}
\end{align*}
\end{proof}

\section{The empirical measure}
\label{sec:mes-emp}

We may assume that the set $\gls*{A} := \{ \chi > R(f) /r \}$ is of positive Lebesgue measure.
By taking a subset of $A$, still of positive Lebesgue measure, we may assume that we have some 
$\gls*{B} > R(f) / r$ such that
$$
\forall x \in A, \chi(x) > b
$$
Now fix 
$\gls*{P} \geq p_0$, 
$\gls*{Beta} > 0$ and 
$\gls*{Eps} > 0$ given by Proposition \autoref{prop:dens-pos}, and fix 
$\gls*{Sigma}$ an $\varepsilon$-bounded reparametrization such that $Leb(A \cap \sigma_*) > 0$.
Note 
$\gls*{Ex} = E_p(x)$ for $x \in \sigma_* \cap A$ and
$\gls*{G}= f^p$.
From Proposition \autoref{prop:dens-pos}, one can take a subset of $A$ of positive Lebesgue measure and assume that for $n$ large enough
$$
\forall x \in A \cap \sigma_*, | (g^n)'(x)| \geq e^{npb} \Rightarrow d_{n}(E(x)) > \beta
$$
We will use the notation $E_n^{M,m}(x)$ defined in the previous section, for $M,n,m \in \bbZ_0^+$.
Lastly, if $J$ is a subset of $I$, we will denote by $\gls*{Leb_norm}$ the normalized Lebesgue measure on $J$, i.e. $Leb_J( \cdot ) = \frac{Leb( \cdot \: \cap J )}{Leb(J)}$.

\subsection{Definition of the empirical measure}

\begin{lem}
\label{lem:def-An}
There is a sequence of positive integers $\gls*{N}$ and measurable sets $(\gls*{A_n} \subset A \cap \sigma_*)_{n \in \mathcal{n}}$ such that
\begin{itemize}
\item[i)] $\frac{1}{n}\log Leb(A_n) \underset{\mathcal{n} \ni n \to \infty}{\longrightarrow} 0$
\item[ii)] For any $m \in \bbZ_0^+$ and for $M$ large enough, we have that $\left ( \int d_n(E_n^{M,m}(x)) \: d Leb_{A_n}(x) \right )_{n \in \mathcal{n}}$ converges to $\beta_m^M > \beta$, where $\beta_m^M \underset{M \to \infty}{\nearrow} \beta^{\infty}$ which does not depend on $m$
\item[iii)] For any $m \in \bbZ_0^+, \limsup\limits_{\mathcal{n} \ni n \to +\infty} M \int d_n(\partial E_n^{M,m}(x)) d Leb_{A_n}(x) \underset{M \to +\infty}{\longrightarrow} 0$
\item[iv)] Define 
$\gls*{EnMm_ronde}$ as the partition whose atoms are maximal sets of positive $Leb_{A_n}$-measure where $x \mapsto E_n^{M,m}(x)$ is constant. Then
$\limsup\limits_{\mathcal{n} \ni n \to +\infty}\frac{1}{n} H_{Leb_{A_n}}(\mathcal{E}_n^{M,m})\underset{M \to +\infty}{\longrightarrow} 0$, where $H_{\lambda}(\mathcal{P})$ denotes the entropy of a partition $\mathcal{P}$ for a measure $\lambda$, defined as $\sum\limits_{P \in \mathcal{P}} - \lambda(P) \log \lambda(P)$.
\end{itemize}
\end{lem}

\begin{proof}
For $n \in \bbZ_0^+$, define $A_n = \{ x \in A \cap \sigma_* \mid d_{n}(E(x)) > \beta \text{ and } |(g^n)'(x)| \geq e^{npb} \}$.
Then let $\mathcal{n}$ be the sequence of integers in the set $\{ n \in \bbZ_0^+ \mid Leb(A_n) \geq \frac{1}{n^2} \}$.
Since the Lyapunov exponent of $f$ on $A$ is larger than $b$, every point of $A$ is in infinitely many $A_n$.
By using Borel-Cantelli's lemma, this implies that $\mathcal{n}$ is infinite.\\

\hspace{1em}- Proof of $i)$:
Simply write $\log Leb(A_n) \geq - 2 \log n$ for $n \in \mathcal{n}$.\\

\hspace{1em}- Proof of $ii)$:
Once we obtain the convergence in $n$, the fact that $( \beta_m^M )_M$ is non-decreasing comes from the fact that $E_n^{M,m}(x)$ is non-decreasing in $M$.
We first prove the convergence in $n$.
Using Cantor's diagonal argument, we have convergence along a subsequence of $\mathcal{n}$ independent of $M$, and we may assume that it is in fact along $\mathcal{n}$.
We prove that $\beta_m^M > \beta$ and that the limit of $(\beta_m^M)_M$ does not depend on $m$ after the proof of item $iii)$, which only uses the convergence of $(\beta_m^M)_M$.\\

\hspace{1em}- Proof of $iii)$: 
Let $M \leq M' \in \bbZ_0^+, n \in \mathcal{n}$ and $x \in A_n$.
From Lemma \autoref{lem:EnM}, we obtain
\begin{align*}
M \frac{\# \partial E_n^{M,m}(x)}{2}
\underset{\text{Lemma \autoref{lem:EnM}}.iv)}&{\leq} \# (E_n^{M',m}(x) \backslash E_n^{M,m}(x)) +  M \frac{\# \partial E_n^{M',m}(x)}{2}\\
\underset{\text{Lemma \autoref{lem:EnM}}.iii)}&{\leq} \# E_n^{M',m}(x) - \# E_n^{M,m}(x) + n \frac{M}{M'} + M
\end{align*}
Dividing by $n$, we get
$$
M d_n(\partial E_n^{M,m}(x)) \leq 2 \left ( d_n(E_n^{M',m}(x)) - d_n(E_n^{M,m}(x)) + \frac{M}{M'} + \frac{M}{n} \right )
$$
By integrating and letting $n \to +\infty$, we get for any $M \leq M' \in \bbZ_0^+$
$$
\limsup\limits_{\mathcal{n} \ni n \to +\infty} M \int d_n(\partial E_n^{M,m}(x)) d Leb_{A_n}(x)
\leq  2 \left ( \beta_m^{M'} - \beta_m^M + \frac{M}{M'} \right )
$$
We conclude by letting $M' \to +\infty$ then $M \to +\infty$ and using the convergence in item $ii)$.\\

\hspace{1em}- Proof of $ii)$, second part:
We prove that $\beta_m^M > \beta$ for $M$ large enough and that $\beta^{\infty}$ does not depend on $m$.
From the definition of $A_n$, we have
$$
\forall n \in \mathcal{n}, \forall x \in A_n, d_n(E(x)) > \beta
$$
Then, we have that $\# E_n^M(x) \geq \# E(x) - \frac{n+M}{M}$
and that $\# E_n^{M,m}(x) \geq \# E_n^M(x) - (m+M) \frac{\# \partial E_n^M(x)}{2}$.
Hence, for any $M' \geq M$, the previous inequalities give
\begin{align*}
\lim\limits_{n \in \mathcal{n}} \int d_n(E_n^{M,m}(x)) d Leb_{A_n}(x) &\geq \limsup\limits_{n \in \mathcal{n}}  \int d_n(E_n^M(x)) - (m+M) \frac{d_n( \partial E_n^M(x))}{2} d Leb_{A_n}(x)\\
&\geq \beta - \frac{1}{M} - \frac{m+M}{M} \left ( \beta_1^{M'} - \beta_1^M + \frac{M}{M'} \right )
\end{align*}
Therefore, letting $M'$ go to infinity, taking $M$ large enough, and changing $\beta$ to a smaller one gives $\beta_m^M \geq \beta$.
To prove that $\beta^{\infty}$ does not depend on $m$, we write that for any $m \leq m'$, we have
$$
\# E_n^{M,m'}(x) \leq \# E_n^{M,m}(x) \leq \# E_n^{M,m'}(x) + (m'+M) \frac{\# \partial E_n^{M}(x)}{2}
$$

\hspace{1em}- Proof of $iv)$:
Notice that any $E \subset \bbZ_0^+$ is uniquely determined by $\partial E$. Let $F_n^{M,m} = \max\limits_{x \in A_n} \# \partial E_n^{M,m}(x)$, so
\begin{align*}
H_{Leb_{A_n}}(\mathcal{E}_n^{M,m}) &\leq \log \sum\limits_{k=0}^{F_n^{M,m}} {n\choose{k}}\\[0.25em]
&\leq \log \sum\limits_{k=0}^{F_n^{M,m}} {n\choose{F_n^{M,m}}} {{F_n^{M,m}}\choose{k}}\\[0.25em]
&= \log \left ( 2^{F_n^{M,m}} {n\choose{F_n^{M,m}}} \right)\\[0.25em]
\underset{\text{(3.1) page \pageref{ineq:3-1}}}&{\leq} F_n^{M,m} \left ( 1 + \log 2 + \log \frac{n}{F_n^{M,m}} \right )
\end{align*}
Then item $ii)$ from Lemma \autoref{lem:EnM} concludes.
\end{proof}

For any $n \in \mathcal{n}$ and $M \in \bbZ_0^+$, we follow the approach from \cite{BCS_cont} and \cite{burguet2024srbPH} and define the empirical probability measure as follows:
$$
\gls*{MunMm} := \frac{\int \sum\limits_{i \in E_n^{M,m}(x)} \delta_{g^i x} \: d Leb_{A_n}(x)}{\int \# E_n^{M,m}(x) \: d Leb_{A_n}(x)}
$$
We will also use the following measure, which corresponds to $\mu_n^M$ with a different normalization
$$
\gls*{NunMm} := \frac{1}{n \beta^{\infty}} \int \sum\limits_{i \in E_n^{M,m}(x)} \delta_{g^i x} \: d Leb_{A_n}(x)
$$
We will also note $\gls*{MunM} = \mu_n^{M,1}$ and $\gls*{NunM} = \nu_n^{M,1}$.

\subsection{Convergence of the empirical measure}

To establish the convergence of these measures, one may simply extract some subsequence that would converge in the weak-$*$ topology.
However, the forthcoming arguments will require some stronger type of convergence.
More precisely, we will use the fact that $\nu_n^{M,m}$ is non-decreasing in $M$ to prove that $\left ( \nu^{M,m} \right )_M$ converges in total variation.\\

Yet, in the following, if we talk about convergence of measures without precising the topology, it will always be in the weak-$*$ topology.

\begin{prop}
\label{prop:CV-mes}
By taking a subsequence of $\mathcal{n}$, we may assume that
\begin{itemize}
\item[i)] For any $M,m \in \bbZ_0^+$, $(\mu_n^{M,m})_n$ and $(\nu_n^{M,m})_n$ converge along $\mathcal{n}$.
\item[ii)] For $M,m \in \bbZ_0^+$, note 
$\gls*{MuM}$ and 
$\gls*{NuM}$ the respective limits from i). Then $(\mu^{M,m})_M$ and $(\nu^{M,m})_M$ both converge to a $g$-invariant borelian probabilty
$\gls*{Mu}$, which does not depend on $m$.
\item[iii)] For any borelian $B$ and $m \in \bbZ_0^+$, $\nu^{M,m}(B) \underset{M \to \infty}{\nearrow} \mu(B)$.
\end{itemize}
\end{prop}

\begin{proof}
Item $i)$ is a consequence of Cantor's diagonal argument.\\

For item $ii)$, we first show that $(\nu^{M,m})_M$ converges. This will give the convergence of $(\mu^{M,m})_M$ as one can apply item $ii)$ from Lemma \autoref{lem:def-An} to the following equality:
$$
\nu_n^{M,m} = \frac{\int d_n(E^{M,m}(x)) dLeb_{A_n}(x)}{\beta^{\infty}} \mu_n^{M,m}
$$
Then, the following inequality gives the convergence of $(\nu^{M,m})_M$:
$$
\forall \psi : I \to \bbR \text{ continuous}, \forall M \leq M',
\left|\int \psi d \nu^{M',m} - \int \psi d \nu^{M,m}\right|
\leq \frac{\beta_m^{M'} - \beta_m^M}{\beta^{\infty}} || \psi ||_{\infty}
$$
Then, the inequality that we used to prove that $\beta^{\infty}$ does not depend on $m$ in item $ii)$ of Lemma \autoref{lem:def-An} gives
$$
\forall \psi : I \to \bbR \text{ continuous}, \forall m \leq m',
\left | \int \psi d \nu^{M,m'} - \int \psi d \nu^{M,m} \right | \leq \limsup\limits_{n \in \mathcal{n}} \frac{|| \psi ||_{\infty}}{2 \beta^{\infty}} (m' + M) \int d_n(\partial E_n^M(x)) d Leb_{A_n}(x) 
$$
So item $iii)$ from Lemma \autoref{lem:def-An} gives that the limit of $(\nu^{M,m})_M$ does not depend on $m$.
Let $\mu$ be the limit of the sequence $(\mu^{M,m})_M$. Since every $\mu_n^{M,m}$ is a borelian probability measure, the limit $\mu$ is a borelian probability. To prove $g$-invariance, we take $M,m \in \bbZ_0^+$ and $n \in \mathcal{n}$ and notice that for $\psi : I \to \bbR$ continuous, we have 
\begin{align*}
\left | \int \psi \: d g_* \nu_n^{M,m} - \int \psi \: d \nu_n^{M,m} \right |
&= \frac{1}{n \beta^{\infty} }\left | \int \sum\limits_{i \in E_n^{M,m}(x)+1} \psi(g^i x) - \sum\limits_{i \in E_n^{M,m}(x)} \psi(g^i x) \: d Leb_{A_n}(x) \right |\\[0.5em]
&\leq || \psi ||_{\infty} \frac{1}{n \beta^{\infty} } \int \# \partial E_n^{M,m}(x) \: dLeb_{A_n(x)}
\end{align*}
Therefore, by using item $ii)$ from Lemma \autoref{lem:EnM}, we get that the above quantity goes to zero when $n$ then $M$ go to infinity, which yields the $g$-invariance of $\mu$.\\

For item $iii)$, we use that when $n$ is fixed, then $(\nu_n^{M,m})_M$ is a non-decreasing sequence of measures. Since the convergence in $n$ is in the weak-$*$ topology, any non-negative continuous function $\psi : I \to \bbR$ satisfies $\int \psi \: d \nu^{M,m} =\nu^{M,m}(\psi) \underset{M \to +\infty}{\nearrow} \mu(\psi) = \int \psi \: d \mu$.
To go from non-negative continuous functions to characteristic functions, one can use the outer regularity of the measures $\nu^{M,m}$.
\end{proof}

\subsection{Properties of the limit measure}
\label{ssec:prop-mu}

Using section \hyperref[sec:tpsHB]{\textbf{4.}}, we show some properties of the limit measure $\mu$ built in the previous section.
Let $\phi_g$ be the geometric potential defined by 
$$
\gls*{Phig} : \begin{cases}
I \to \bbR \cup \{ - \infty \}\\
x \mapsto \log |g'(x)|
\end{cases}
$$

\begin{prop}
\label{prop:phi-integ}
The function $\phi_g$ is $\mu$-integrable
\end{prop}

\begin{proof}
Since $\phi_g \leq \log || g' ||_{\infty}$, we only have to prove $\int \phi_g^- d \mu < +\infty$, where $\phi_g^- := - \min (\phi_g, 0)$.
For $k \in \bbZ_0^+$, note
$$
\phi_{g,k} = \min(k, \phi_g^-) : I \to \bbR_+ \underset{k \to +\infty}{\nearrow} \phi_g^-
$$
By monotone convergence and continuity of $\phi_{g,k}$, we get
\begin{align*}
\int \phi_g^- d \mu = \lim\limits_{k \to +\infty} \uparrow \int \phi_{g,k} \: d \mu
&= \lim\limits_{k \to +\infty} \lim\limits_{M \to +\infty} \lim\limits_{\mathcal{n} \ni n \to \infty} \int \phi_{g,k} \: d \mu_n^M\\[0.5em]
&= \lim\limits_{k \to +\infty} \lim\limits_{M \to +\infty} \lim\limits_{\mathcal{n} \ni n \to \infty} \frac{\int \sum\limits_{i \in E_n^M(x)} \phi_{g,k}(g^ix) d Leb_{A_n}(x)}{\int \# E_n^M(x) d Leb_{A_n}(x)}
\end{align*}
We then estimate the terme inside the integral.
For $k \in \bbZ_0^+$, $M \in \bbZ^+$, $n \in \mathcal{n}$ and $x \in A_n$, we have

\begin{align*}
\sum\limits_{i \in E_n^M(x)} \phi_{g,k}(g^ix)
&\leq \sum\limits_{i \in E_n^M(x)} -\min(0, \phi_g(g^i x))\\
&\leq \sum\limits_{i \in E_n^M(x)} \left ( - \phi_g(g^i x) + \log || g' ||_{\infty} \right )\\
\underset{\text{Lemma \autoref{lem:tpsHB}.}ii)}&{\leq} 0 + \# E_n^M(x) \log || g' ||_{\infty}
\end{align*}
where we used item $ii)$ from Lemma \autoref{lem:tpsHB} on every connected component of $E_n^M(x)$.
This then leads to $\int \phi_g^- d \mu \leq \log || g' ||_{\infty} < +\infty$.
\end{proof}

\begin{prop}
\label{prop:expo-pos}
We have $\chi_g > 0$ $\mu$-almost everywhere and $\int \phi_g d \mu = \int \chi_g d \mu$.
\end{prop}

\begin{proof}
To prove that $\int \phi_g d \mu = \int \chi_g d \mu$, we use Birkhoff's ergodic theorem for the potential $\phi_g \in \mathbb{L}^1(\mu)$.
Then, to prove that $\chi_g > 0$ $\mu$-almost everywhere, we adapt the proof of Lemma 6 from \cite{MR4703423}.
We will in fact prove that $\chi_g \geq \log 10$ $\mu$-almost everywhere.
For $m \in \bbZ_0^+$, we note $\phi_g^m =  \sum\limits_{k \in [\![ 0 ; m-1 ]\!]} \phi_g \circ g^k$.
For $M \in \bbZ_0^+$, let $K_M = \{ x \in I \mid \exists m \in [\![ 1 ; M ]\!], \phi_g^m(x) \geq m \log 10 \}$.
Although $\phi_g$ is not continuous everywhere, the set $K_M$ is closed because $\phi_g$ is not continuous only on the critical set $\mathcal{C}_g$.
Then, for $x \in \sigma_*, M \in \bbZ_0^+$ and $n \in \mathcal{n}$, we have that any $k \in E_n^M(x)$ satisfies $g^k x \in K_M$.
Indeed, if $k \in E_n^M(x)$, then there exists $l \in E(x)$ such that $1 \leq l-k \leq M$, hence item $i)$ from Lemma \autoref{lem:tpsHB} gives
$$
\phi_g^{l-k}(g^k x) \geq (l-k) \log 10
$$
Therefore $\mu_n^M(K_M) = 1$. Since $K_M$ is closed, letting $n$ go to infinity gives $\mu^M(K_M) = 1$, which we rather write as $\nu^M(K_M) = \frac{\beta^M}{\beta^{\infty}}$, where $\beta^M \underset{M \to +\infty}{\nearrow} \beta^{\infty}$ as stated in item $ii)$ from Lemma \autoref{lem:def-An}.
Hence, if we let $K = \bigcup\limits_{M \geq 1} K_M$, then item $iii)$ from Proposition \autoref{prop:CV-mes} gives $\mu(K) = 1$.
Since $\mu$ is $g$-invariant, we also have that $\mu\left ( \bigcap\limits_{k \geq 0} g^{-k} K \right ) = 1$.
In other words, for $\mu$-almost every $x \in I$, we have that
$$
\forall k \geq 0, \exists m \geq 1, \phi_g^m(g^k x) \geq m \log 10
$$
By applying this iteratively, we obtain a strictly non-decreasing sequence of integers $(m_k)$ such that
\begin{align*}
&\phi_g^{m_0}(x) \geq m_0 \log 10\\
&\phi_g^{m_1 - m_0}(g^{m_0} x) \geq (m_1 - m_0) \log 10\\
&...\\
&\phi_g^{m_{k+1} - m_k}(g^{m_k} x) \geq (m_{k+1} - m_k) \log 10
\end{align*}
By summing these inequalities, we obtain that for any $k$, $\phi_g^{m_k}(x) \geq m_k \log 10$, hence $\chi_g(x) \geq \log 10$.

\end{proof}

\section{Entropy of the empirical measure}
\label{sec:entropie}

Let us sum up what we have done so far.
We noted $A = \{ \chi > R(f) / r \}$ and fixed $g = f^p$ some iterate of $f$, $\varepsilon > 0, \beta > 0$ and $\sigma$ an $\varepsilon$-bounded reparametrization such that $Leb(\sigma_* \cap A) > 0$.
We applied the Reparametrization Lemma to $\sigma$ and $g$ and got a family of reparametrizations organized as a tree with unbounded degree.
For $x \in \sigma_* \cap A$, we defined $E(x)$ the set of geometric times of $x$, whose density is larger than $\beta$ for Lebesgue-almost every $x$.
Then, using two parameters $M$ and $m$ in $\bbZ_0^+$, we defined the sets $E^{M,m}(x)$ so that the orbit of $x$ at these times does not get too close to critical points.
By integrating those pieces of orbits up to time $n$, we got an empirical measure $\mu_n^{M,m}$.
Then, by letting $n \to +\infty$ along some sequence $\mathcal{n}$ then by letting $M \to +\infty$, we got a $g$-invariant measure $\mu$, and we proved that $\mu$ is hyperbolic and that $\log |g'|$ is $\mu$-integrable.
We now wish to prove that $\mu$ is absolutely continuous with respect to the Lebesgue measure. By using the entropy characterization given by Theorem \autoref{th:form-entropie}, we will in fact prove that $\mu$ satisfies the entropy formula $h_{\mu}(g) = \int \log |g'| d \mu$.\\

In this section, we give a lower bound of the entropy of the empirical measure $\mu_n^{M,m}$ for a specific countably infinite partition $\mathcal{P}_q$ that depends on a parameter $q \in \bbZ^+$.

\subsection{First entropy estimates}
\label{ssec:entropie1}

We start this section with a slightly modified version of Lemma 5 from \cite{MR4701884}, which gives a lower bound of the entropy of an empirical measure for a finite measurable partition.
We state a version for countably infinite measurable partitions for which the measure only sees a finite number of atoms. More precisely, if $\mathcal{P}$ is a measurable partition and $\mu$ is a borelian measure, then we ask for the following set to be finite:
$$
\gls*{Partition_mu} := \{ P \in \mathcal{P} \mid \mu(P) > 0 \}
$$
Also, what we call a measurable partition depends on the measure:

\begin{definition}[Measurable partition]
If $(X, \mathcal{A},\lambda)$ is a probability space and $\mathcal{R}$ is a countable collection of subsets of $X$, we say that $\mathcal{R}$ is a measurable partition for $\lambda$ if
\begin{itemize}
\item[i)] for $R \neq R' \in \mathcal{R}$, $\lambda(R \cap R') = 0$,
\item[ii)] the union $\bigcup\limits_{R \in \mathcal{R}} R$ has full $\lambda$-measure.
\end{itemize}
\end{definition}

The important remark is that when a measure $\lambda$ is not preserved by a dynamical system $T$, and if $\mathcal{R}$ is a measurable partition for $\lambda$, then $T^{-1} \mathcal{R}$ may not be a measurable partition for $\lambda$.

\begin{lem}
\label{lem:Misiurewicz}
Let $\mathcal{R}$ be a countable measurable partition of a probabilty space $(X, \lambda)$. Let $T : X \to X$ be a measurable transformation, which may not preserve $\lambda$. Let $F$ be a finite subset of $\bbZ_0^+$.
Note $\lambda^F = \frac{1}{\# F} \sum\limits_{k \in F} T^k_* \lambda$.
Assume that for every $i \in F$, $T^{-i} \mathcal{R}$ is a measurable partition for $\lambda$ and that $\mathcal{R}_{\lambda^F}$ is finite.
For $m \in \bbZ^+$, if we note $\gls*{Partition_m}=\bigvee\limits_{k=0}^{m-1} T^{-k} \mathcal{R}$ and $\gls*{RF} = \bigvee\limits_{k \in F} T^{-k} \mathcal{R}$, then we have
$$
\frac{1}{m} H_{\lambda^F} (\mathcal{R}^m) \geq \frac{1}{\# F} H_{\lambda}(\mathcal{R}^F) - m \log (\# \mathcal{R}_{\lambda^F}) \frac{\# \partial F}{\# F}
$$
\end{lem}
\begin{proof}
When $F = [\![ a; b ]\!]$, we use an argument from Misiurewicz's proof of the variational principle \cite{Misiurewicz1976}.
We start with the case $m \in [\![ 1 ; b-a ]\!]$, we let $r \in [\![ 0 ; m-1 ]\!]$ and note $j_r = \lfloor \frac{b-a - r}{m} \rfloor$.
We have
\begin{align*}
\mathcal{R}^F &= \bigvee\limits_{i=a}^b T^{-i} \mathcal{R}
= T^{-a} \bigvee\limits_{i=0}^{b-a} T^{-i} \mathcal{R}\\
&=T^{-a} \left [ \left ( \bigvee\limits_{j=0}^{j_r - 1} T^{-(mj + r)} \mathcal{R}^m \right )
\vee \left ( \bigvee\limits_{i=0}^{r-1} T^{-i} \mathcal{R} \right )
\vee \left ( \bigvee\limits_{i=r+j_r m}^{b-a} T^{-i} \mathcal{R} \right ) \right ]
\end{align*}
Then, $T^{-i} \mathcal{R}$ is a measurable partition when $i \in F$, so its entropy is well defined and we have
\begin{align*}
H_{\lambda}(\mathcal{R}^F) 
\leq \sum\limits_{j=0}^{j_r - 1} H_{\lambda} \left ( T^{-(mj + r+a)} \mathcal{R}^m \right )
+ \sum\limits_{i=0}^{r-1} H_{\lambda} \left ( T^{-(i+a)} \mathcal{R} \right )
+ \sum\limits_{i=r+j_r m}^{b-a} H_{\lambda} \left ( T^{-(i+a)} \mathcal{R} \right )
\end{align*}
Then, we will use the following fact:
\begin{align}
\forall i \in F, \forall A \text{ measurable}, \left ( \lambda(T^{-i} A) > 0 \Rightarrow \lambda^F(A) > 0 \right )
\end{align}
This implies, for $i \in F$, that
$$
H_{\lambda}(T^{-i} \mathcal{R}) \leq \log \# \mathcal{R}_{\lambda^F}
$$
Indeed, when $T^{-i} R$ is of positive $\lambda$-measure, then $R \in \mathcal{R}_{\lambda^F}$.
Therefore, by using $r + (b-a) - (r+ j_r m) + 1 \leq 2m$, we get
$$
H_{\lambda}(\mathcal{R}^F) \leq 2m \log \# \mathcal{R}_{\lambda^F} + \sum\limits_{j=0}^{j_r - 1} H_{\lambda} \left ( T^{-(mj + r+a)} \mathcal{R}^m \right )
$$
Summing over $r$ then gives
$$
m H_{\lambda}(\mathcal{R}^F) \leq \sum\limits_{i=a}^{b-m} H_{\lambda}(T^{-i} \mathcal{R}^m) + 2m^2 \log \# \mathcal{R}_{\lambda^F}
\leq 2m^2 \log \# \mathcal{R}_{\lambda^F} + \sum\limits_{i \in F} H_{\lambda}(T^{-i} \mathcal{R}^m)
$$
Then, if $m \geq |b-a|$, we have $H_{\lambda}(\mathcal{R}^F) \leq \# F \log \# \mathcal{R}_{\lambda^F} \leq m \log \# \mathcal{R}_{\lambda^F}$. So the above inequality still holds.
Hence, by using the concavity of $\lambda \mapsto H_{\lambda}(\mathcal{R}^m)$, we get
\begin{align*}
\frac{1}{m} H_{\lambda^F}(\mathcal{R}^m) \geq \frac{1}{m \# F} \sum\limits_{i \in F} H_{\lambda}(T^{-i} \mathcal{R}^m)
\geq \frac{1}{\# F} \left ( H_{\lambda}(\mathcal{R}^F) - 2m \log \# \mathcal{R}_{\lambda^F} \right )
\end{align*}
Then, if $F$ is a union of intervals $F_k = [\![ a_k ; b_k ]\!]$ with $k \in [\![ 1 ; N ]\!]$, using the concavity of $\lambda \mapsto H_{\lambda}(\mathcal{R}^m)$ gives
\begin{align*}
\frac{1}{m} H_{\lambda^F}(\mathcal{R}^m) &\geq \frac{1}{m} \sum\limits_{k=1}^N \frac{\# F_k}{\# F} H_{\lambda^{F_k}}(\mathcal{R}^m)\\
\underset{\text{previous case}}&{\geq} \: \sum\limits_{k=1}^N \frac{1}{\# F} \left ( H_{\lambda}(\mathcal{R}^{F_k}) - 2m \log \# \mathcal{R}_{\lambda^{F_k}} \right )\\
&\geq \frac{H_{\lambda}(\mathcal{R}^F)}{\# F} - \sum\limits_{k=1}^N \frac{2m}{\# F} \log \# \mathcal{R}_{\lambda^{F_k}}
\end{align*}
Notice that $2N = \# \partial F$ and consider the following fact, analog to $(6.1)$:
$$
\forall k \in [\![ 1 ; N ]\!], \forall A \text{ measurable}, \left ( \lambda^{F_k}(A) > 0 \Rightarrow \lambda^F(A) > 0 \right )
$$
This leads to $\# \mathcal{R}_{\lambda^{F_k}} \leq \# \mathcal{R}_{\lambda^F}$, which concludes.
\end{proof}

For the rest of this section, fix $m \in \bbZ^+, M \in \bbZ_0^+, n \in \mathcal{n}$, and recall that
$\mathcal{E}_n^{M,m}$ is the partition of $A \cap \sigma_*$ whose atoms are the sets $\{ E_n^{M,m}(x) = E \}$ where $E \subset \bbZ_0^+$ are such that the associated atoms are of positive $Leb_{A_n}$-measure.
To shorten notations, we will denote by $E$ both the atom of $\mathcal{E}_n^{M,m}$ and the value $E_n^{M,m}(x)$ for $x$ in this atom. Then for $E \in \mathcal{E}_n^{M,m}$, we note
$$
\gls*{Mu_E} = \frac{Leb_{A_n}(E \cap \cdot)}{Leb_{A_n}(E)} \; \; \; \text{and} \; \; \; \mu^E = \frac{1}{\# E} \sum\limits_{k \in E} g^k_* \mu_E
$$
Let $\gls*{J}$ be the partition of $I$ into monotone branches of $g$ as defined in Definition \autoref{def:monot-branche}.

\begin{rmq}
The collection $\mathcal{J}$ is a measurable partition for $\mu$, because its atoms are disjoint and cover $I \backslash \mathcal{C}_g$ which is of full $\mu$-measure since $\log |g'|$ is $\mu$-integrable (Proposition \autoref{prop:phi-integ}).
Moreover, when we consider the partition $\mathcal{J}^n$, its atoms may be of the form $[a;b[, ]a;b]$, or $]a;b[$. Therefore, it is equal to the partition into monotone branches of $g^n$ Lebesgue-almost everywhere, but it may not be equal everywhere.
\end{rmq}

We first show that $H_{\mu_E}(\mathcal{J}^E) < +\infty$ for $E \in \mathcal{E}_n^{M,m}$, this will ensure the computations of the rest of the section hold. We write
$$
E = \bigcup\limits_{j = 1}^d [\![ a _j ; b_j [\![
$$
So we get
\begin{align*}
H_{\mu_E}(\mathcal{J}^E) &\leq \log \# \{ J \in \mathcal{J}^E \mid \mu_E(J) > 0 \}\\
&\leq \log \# \{ J \in \mathcal{J}^E \mid J \cap A_n \cap E \neq \emptyset \}\\
&\leq \sum\limits_{j=1}^d \log \# \{ J \in \mathcal{J}^{[\![ a_j ; b_j [\![} \mid J \cap A_n \cap E \neq \emptyset \}
\end{align*}
But when $y$ is in $J \cap A_n \cap E$, then item $ii)$ of Lemma \autoref{lem:tpsHB} gives
$$
|(g^{b_j - a_j})'(g^{a_j} y)| \geq 1
$$
Therefore the entropy is finite, because $g^{b_j - a_j}$ is $\mathcal{C}^1$, so the number of monotone branches where the derivative reaches some fixed positive value is finite.
We also mention that in the $\mathcal{C}^r$ case, we have an estimate of this number.
This is the statement of the following lemma, though we will only use it in section \hyperref[ssec:entropie2]{\textbf{6.b.}}

\begin{lem}
\label{lem:brch-mon-DM}
Let $g : I \to I$ be a $\mathcal{C}^r$ map where $I$ is the interval $[0
;1]$ or the circle $\bbT^1$, and $r>1$.
Note $r' = \min(r,2)$.
For $s > 0$, the number of monotone branches of $g$ where $|g'|$ takes at least the value $s$ is less than $C(r',g) s^{-\frac{1}{r'-1}} + 1$.
\end{lem}

\begin{proof}
Because $g$ is $\mathcal{C}^{r'}$, we have the following inequality:
$$
\forall x,y \in I, |g'(x)-g'(y)| \leq || d^{r'} g ||_{\infty} |x-y|^{r' - 1}
$$
So whenever $|g'(x)| \geq s$, then $g'$ cannot vanish in the $|| d^{r'} g ||_{\infty}^{- \frac{1}{r' - 1}} s^{\frac{1}{r' - 1}}$-neighborhood of $x$.
Therefore, there are at most $|| d^{r'} g ||_{\infty}^{\frac{1}{r' - 1}} s^{- \frac{1}{r' - 1}} + 1$ intervals where $|g'|$ reaches $s$ and whose extremities are critical points or points of $\partial I$.
Notice that the $+1$ is only needed if $I$ is the interval.
Now, recall that if $I$ is the circle, then an interval whose extremities are critical points may not be a monotone branch (see section \hyperref[ssec:diff-int-cercle]{\textbf{2.d.}}).
So we may divide the previous intervals into at most $|| g' ||_{\infty}$ pieces.
This requires multiplying our previous bound by $|| g' ||_{\infty}$ if $I$ is the circle, which concludes.
\end{proof}
\vspace{1em}

We now give the entropy estimate of the measure $\mu_n^{M,m}$ for a partition $\mathcal{P}$.

\begin{prop}
\label{prop:entrop-estim}
Let $\mathcal{P}$ be a countable measurable partition of $I$ for $\mu$ such that $\mathcal{P}_{\mu^E}$ is finite for every $E \in \mathcal{E}_n^{M,m}$ and such that $g^{-i} \mathcal{P}$ is a measurable partition for $\mu_E$, when $i \in E$.
Then for any $m,m' \in \bbZ^+, M \in \bbZ_0^+$ and $n \in \mathcal{n}$, we have
\begin{align*}
\frac{1}{m'} \left ( \int \# E_n^{M,m}(x) d Leb_{A_n}(x) \right ) H_{\mu_n^{M,m}}(\mathcal{P}^{m'})
&\geq \sum\limits_{E \in \mathcal{E}_n^{M,m}} \int_E - \log Leb(\mathcal{P}_{x}^E \cap E \cap A_n) d Leb_{A_n}(x)\\
&\hspace{3em}+ Leb_{A_n}(E) \log Leb_{A_n}(E)\\
&\hspace{3em}+ Leb_{A_n}(E) \log Leb(A_n)\\
&\hspace{1em}- m' \sum\limits_{E \in \mathcal{E}_n^{M,m}} Leb_{A_n}(E) \log ( \# \mathcal{P}_{\mu^E}) \# \partial E
\end{align*}
where $\mathcal{P}_x^E$ is the element of the partition $\mathcal{P}^E$ that contains $x$.
\end{prop}

Let us dissect this statement. The right hand side contains four terms, our goal will be to show that when $\mathcal{P}$ is a well-chosen partition, then the first term gives the Lyapunov exponent 
$\gls*{Chi_gmu} = \int \chi_g \: d \mu$ and the other three are negligeable.
In section \hyperref[ssec:partition]{\textbf{6.b.}}, we define this partition, show that its $\mu$-entropy is finite and that the fourth term is negligeable.
The second (resp. third) term is always negligeable by item $iv)$ (resp. $i)$) from Lemma \autoref{lem:def-An}.
Then, we estimate the first term in section \hyperref[ssec:gibbs]{\textbf{6.c.}}, by establishing a Gibbs inequality.\\

\begin{proof}[Proof of Proposition \autoref{prop:entrop-estim}]
We first consider the equality
$$
\sum\limits_{E \in \mathcal{E}_n^{M,m}} \# E Leb_{A_n}(E) \mu^E = \left ( \int \# E_n^{M,m}(x) d Leb_{A_n}(x) \right ) \mu_n^{M,m}
$$
By concavity of $\mu \mapsto H_{\mu} (\mathcal{P}^{m'})$, this equality gives
\begin{align*}
\frac{1}{m'} \left ( \int \# E_n^{M,m}(x) d Leb_{A_n}(x) \right ) H_{\mu_n^{M,m}}(\mathcal{P}^{m'}) &\geq \sum\limits_{E \in \mathcal{E}_n^{M,m}} Leb_{A_n}(E) \frac{\# E}{m'} H_{\mu^E} (\mathcal{P}^{m'})\\
\underset{\text{Lemma } \autoref{lem:Misiurewicz}}&{\geq} \sum\limits_{E \in \mathcal{E}_n^{M,m}} Leb_{A_n}(E) \left ( H_{\mu_E}(\mathcal{P}^E) - m' \log (\# (\mathcal{P}_{\mu^E})) \# \partial E \right )\\
&=\sum\limits_{E \in \mathcal{E}_n^{M,m}} Leb_{A_n}(E) \int - \log \mu_E(\mathcal{P}_x^E) d \mu_E(x)\\
&\hspace{1em}- m' \sum\limits_{E \in \mathcal{E}_n^{M,m}} Leb_{A_n}(E) \log (\# (\mathcal{P}_{\mu^E})) \# \partial E
\end{align*}
We conclude by using the following equality
$$
\mu_E(\mathcal{P}_x^E) = \frac{Leb_{A_n}(E \cap \mathcal{P}_x^E)}{Leb_{A_n}(E)}
= \frac{Leb(A_n \cap E \cap \mathcal{P}_x^E)}{Leb(A_n) Leb_{A_n}(E)}
$$
which gives
\begin{align*}
Leb_{A_n}(E) \int - \log \mu_E(\mathcal{P}_x^E) d \mu_E(x)
&= Leb_{A_n}(E)  \int_E - \log \frac{Leb(A_n \cap E \cap \mathcal{P}_x^E)}{Leb(A_n) Leb_{A_n}(E)} \frac{d Leb_{A_n}(x)}{Leb_{A_n}(E)}\\
& = \int_E  - \log Leb(A_n \cap E \cap \mathcal{P}_x^E) d Leb_{A_n}(x)\\
&\hspace{1em}+ Leb_{A_n}(E) \log Leb(A_n) + Leb_{A_n}(E) \log Leb_{A_n}(E)
\end{align*}
\end{proof}

\subsection{Definition of the partition $\mathcal{P}_q$}
\label{ssec:partition}

Let $q \in \bbZ^+$ be a parameter, we build a partition $\mathcal{Q}_q$ as follows
\begin{itemize}
\item[1)] for $k \in \bbZ$, note 
$\gls*{Iqk} = ] k/q ; (k+1)/q] + a$ where $a \in \bbR$ does not depend on $k$ and is chosen below,
\item[2)] for $k \in \bbZ$, note 
$\gls*{Qqk} = \{ x \in I \mid \log |g'(x)| \in I_{q,k} \}$.
\end{itemize}
Then we define the partition
$\gls*{Qq} = \{ Q_{q,k} \mid k \in \bbZ \}$. Hence, on each atom of $\mathcal{Q}_q$, the value of $\log |g'|$ does not vary by more than $1/q$.
Then, to choose $a$, recall that $\mathcal{J}$ is the measurable partition for $\mu$ into monotone branches of $g$. We choose $a \in \: ]- 1 /q ; 0[$ such that the border of $\mathcal{J} \vee \mathcal{Q}_q$ has zero $\mu$ and $\mu^{M,m}$-measure, for any $M,m \in \bbZ_0^+$.
Also notice that $k/q \in I_{q,k}$ for any $k \in \bbZ$.
We define
$$
\gls*{Pq} = \mathcal{J} \vee \mathcal{Q}_q
$$
This may be an countably infinite partition, and the purpose of this section is to prove the following statement:

\begin{prop}
\label{prop:partition}
The collection $\mathcal{P}_q$ is a measurable partition for $\mu$ satisfying:
\begin{itemize}
\item[1)] $H_{\mu}(\mathcal{P}_q) < + \infty$
\item[2)] For $E \in \mathcal{E}_n^{M,m}$, we have $\# (\mathcal{P}_q)_{\mu^E} < + \infty$
\item[3)] For $E \in \mathcal{E}_n^{M,m}$ and $i \in E$, the collection $g^{-i} \mathcal{P}_q$ is a measurable partition of $I$ for $\mu_E$
\item[4)] $\limsup\limits_{\mathcal{n} \ni n \to +\infty} \frac{1}{n}\sum\limits_{E \in \mathcal{E}_n^{M,m}} Leb_{A_n}(E) \# \partial E \log \# ((\mathcal{P}_q)_{\mu^E}) \underset{M \to +\infty}{\longrightarrow} 0$
\end{itemize}
\end{prop}

The core of the argument is that $\log |g'|$ is $\mu$-integrable (Proposition \autoref{prop:phi-integ}).
We point out that the proof of item $1)$ uses some ideas from the proof of Lemma 2 from Mañé's proof of Pesin's formula \cite{Mane_Pesin}. We will also use the following technical lemma:

\begin{lem}[Lemma 1 from \cite{Mane_Pesin}]
\label{lem:SETE-tech}
For $x_k \in [0;1], k \in \bbZ$, we have
$$
\sum\limits_{k \in \bbZ} - x_k \log x_k \leq \sum\limits_{k \in \bbZ} |k| x_k + c_0
$$
where $c_0 = 4(e(1 - e^{-1/2}))^{-1}$
\end{lem}

\begin{proof}[Proof of Proposition \autoref{prop:partition}]
By definition, atoms of $\mathcal{P}_q$ cover $I \backslash \mathcal{C}_g$, which is of full $\mu$-measure since $\log |g'| \in \mathbb{L}^1(\mu)$ (Proposition \autoref{prop:phi-integ}). Then, atoms of $\mathcal{Q}_q$ and $\mathcal{J}$ are disjoint, so $\mathcal{P}_q$ is a measurable partition for $\mu$.
The argument to prove $3)$ is similar:\\

\hspace{1em}- Proof of $3)$:
Notice that $g^{-i} \mathcal{C}_g \subset \mathcal{C}_{\infty}$.
Therefore, it is enough to show that $\mu_E(\mathcal{C}_{\infty}) = 0$, which is true since $A \cap \mathcal{C}_{\infty} = \emptyset$.
\vspace{1em}

\hspace{1em}- Proof of $1)$: We first show that $\mathcal{Q}_q$ is of finite $\mu$-entropy.
We will then show that $\mathcal{J}$ as well, which is the part inspired by Lemma 2 from \cite{Mane_Pesin}.
We have
\begin{align*}
H_{\mu}(\mathcal{Q}_q)
&= \sum\limits_{k \in \bbZ} - \mu(Q_{q,k}) \log \mu(Q_{q,k})\\
\underset{\text{Lemma } \autoref{lem:SETE-tech}}&{\leq} c_0 + \sum\limits_{k \in \bbZ} |k| \mu(Q_{q,k})\\
&= c_0 + \sum\limits_{k \in \bbZ} \int |k| \mathds{1}_{\{ x \in I \; \mid \; \log |g'(x)| \in I_{q,k} \}} d \mu(x)\\
&\leq c_0 + \sum\limits_{k \in \bbZ} \int (q |\log |g'(x)|| + 1) \mathds{1}_{\{ x \in I \; \mid \; \log |g'(x)| \in I_{q,k}\}} d \mu(x)\\
&= c_0 + 1 + q \int |\log |g'|| d \mu\\
\underset{\text{Proposition } \autoref{prop:phi-integ}}&{<} +\infty
\end{align*}
For $\mathcal{J}$, we will also use Lemma \autoref{lem:SETE-tech} and Proposition \autoref{prop:phi-integ}. 
We first examine the specificities of $\mathcal{J}$.
Recall that to define $\mathcal{J}$ (Definition \autoref{def:monot-branche}), we first divided $I$ according to the critical points of $g$, then we divided again some of these pieces into at most $|| g' ||_{\infty}$ pieces to ensure injectivity in case $I$ is the circle.
Therefore, to show that $\mathcal{J}$ has finite entropy, we may assume that $\mathcal{J}$ only comes from the division of $I$ according to the critical points of $g$.
However, this change of definition only holds for this proof.
Let $r' = \min(r,2)$, we have
$$
\forall x \in I, |g'(x)| \leq dist(x, \mathcal{C}_g)^{r'-1} || d^{r'} g||_{\infty}
$$
Therefore, if $x \notin \mathcal{C}_g$ and $\rho(x)$ denotes the size of the monotone branch containing $x$, then
$$
\rho(x) \geq \left ( \frac{|g'(x)|}{|| d^{r'} g ||_{\infty}} \right )^{\frac{1}{r' - 1}}
$$
Let $\mathcal{J}_n = \{ J \in \mathcal{J} \mid | \log \text{ diam } J| \in [ n ; n+1[ \} = \{ J_k^{(n)} \}_{k \in N_n}$ and $J_n = \bigcup\limits_{k \in N_n} J_k^{(n)}$.
Thus we have
$$
H_{\mu}(\mathcal{J}) = \sum\limits_{J \in \mathcal{J}} - \mu(J) \log \mu(J)\\
= \sum\limits_{n \in \bbZ_0^+} \sum\limits_{k \in N_n} - \mu(J_k^{(n)}) \log \mu(J_k^{(n)})
$$
Then, concavity of $\log$ implies, for any $n \in \bbZ_0^+$, that
\begin{align*}
\sum\limits_{k \in N_n} - \mu(J_k^{(n)}) \log \mu(J_k^{(n)})
= \mu(J_n)\sum\limits_{k \in N_n} \frac{\mu(J_k^{(n)})}{\mu(J_n)} \log \frac{1}{\mu(J_k^{(n)})}
\leq \mu(J_n) \log \frac{\# N_n}{\mu(J_n)}
\end{align*}
We therefore get
$$
H_{\mu}(\mathcal{J}) \leq \sum\limits_{n \in \bbZ_0^+} \mu(J_n) \log \# N_n - \mu(J_n) \log \mu(J_n)
$$
Then, any $J \in \mathcal{J}_n$ has length larger than $e^{-(n+1)}$, so $\# N_n \leq e^{n+1}$, so Lemma \autoref{lem:SETE-tech} implies that
$$
H_{\mu}(\mathcal{J}) \leq c_0 + 2 \sum\limits_{n \in \bbZ_0^+} n \mu(J_n)
$$
But notice that $J_n = \{ x \in I \mid | \log \rho(x)| \in [ n ; n+ 1[ \}$, hence
$$
H_{\mu}(\mathcal{J}) \leq c_0 + 2 \int |\log \rho(x)| d \mu(x) \leq
c_0 + 2 \int \frac{1}{r'-1} \left | \log |g'(x)| - \log || d^{r'} g ||_{\infty} \right | d \mu(x)
\underset{\text{Proposition } \autoref{prop:phi-integ}}{<} +\infty
$$
\vspace{1em}

\hspace{1em}- Proof of $2)$:
Recall the notation from the previous section: $(\mathcal{P}_q)_{\mu^E} = \{ P \in \mathcal{P}_q \mid \mu^E(P) > 0 \}$, where $E \in \mathcal{E}_n^{M,m}$.
Recall that $\mu^E = \frac{1}{\# E} \sum\limits_{a \in E} g^a_* (Leb_{A_n})_E$, thus
$$
\mu^E(P) > 0 \Rightarrow \exists a \in E, \exists x \in A_n \cap E, g^a x \in P
$$
Since $\mathcal{P}_q = \mathcal{Q}_q \vee \mathcal{J}$, we only have to show that $\# (\mathcal{Q}_q)_{\mu^E} < +\infty$ and $\# \mathcal{J}_{\mu^E} < +\infty$, independently.
Although we only have to show finiteness, we provide an upper bound that will be useful to prove item $4)$.
For $(\mathcal{Q}_q)_{\mu^E}$, we have
\begin{align*}
\# (\mathcal{Q}_q)_{\mu^E} &\leq \# \{ k \in \bbZ \mid \exists a \in E, \exists x \in A_n \cap E, \log |g'(g^a x)| \in I_{q,k} \}\\
&\leq \# \{ k \in \bbZ \mid \exists a \in E, \exists x \in A_n \cap E, k-1 < q \log |g'(g^a x)| < k+1 \}
\end{align*}
But when $x \in A_n \cap E$ and $a \in E$, item $iii)$ from Lemma \autoref{lem:tpsHB} gives $\log |g'(g^a x)| \geq - M \log || g' ||_{\infty}$, so
$$
\# (\mathcal{Q}_q)_{\mu^E} \leq q \log || g' ||_{\infty} + 1 - ( - q M \log || g' ||_{\infty} - 1) + 1 < + \infty
$$
Then for $\mathcal{J}_{\mu^E}$, using once again item $iii)$ from Lemma \autoref{lem:tpsHB} gives
\begin{align*}
\# \mathcal{J}_{\mu^E} &\leq \# \{ J \in \mathcal{J} \mid \exists a \in E, \exists x \in A_n \cap E, g^a x \in J \}\\
&\leq \# \left \{ J \in \mathcal{J} \mid \exists y \in J, |g'(y)| \geq  \frac{1}{|| g' ||_{\infty}^M} \right \}\\
\underset{\text{Lemma } \autoref{lem:brch-mon-DM}}&{\leq} C(r',g) || g' ||_{\infty}^{\frac{M}{r'-1}} + 1\\
&< + \infty
\end{align*}

\hspace{1em}- Proof of $4)$:
Notice that $\# ((\mathcal{P}_q)_{\mu^E}) \leq \# ((\mathcal{Q}_q)_{\mu^E}) \times \# \mathcal{J}_{\mu^E}$.
By using the upper bounds proved for item $2)$, and by choosing $M$ large enough, we get
\begin{align*}
\sum\limits_{E \in \mathcal{E}_n^{M,m}} Leb_{A_n}(E) \# \partial E \log \# ((\mathcal{P}_q)_{\mu^E}) 
&\leq \sum\limits_{E \in \mathcal{E}_n^{M,m}} Leb_{A_n}(E) \# \partial E \log \left [ 2M q \log || g' ||_{\infty} \times  \left ( C(r',g) || g' ||_{\infty}^{M / (r' - 1)} + 1 \right ) \right ]\\[0.5em]
&\leq \left (\frac{M}{r' - 1} \log (q C(r',g) ) \right ) \int \# \partial E_n^{M,m}(x) dLeb_{A_n}(x)
\end{align*}
We conclude with $iii)$ from Lemma \autoref{lem:def-An}, stating that $\limsup\limits_{\mathcal{n} \ni n \to +\infty} M \int d_n(\partial E_n^{M,m}(x)) d Leb_{A_n}(x) \underset{M \to +\infty}{\longrightarrow} 0$.
\end{proof}

\subsection{Gibbs inequality}
\label{ssec:gibbs}

Let $n \in \mathcal{n}, M \in \bbZ_0^+$ and $E \in \mathcal{E}_n^{M,m}$. Recall that $\phi_g = \log |g'|$, and note
$$
\gls*{PhigE} = \sum\limits_{k \in E} \phi_g \circ g^k
$$
In this section, we deal with the main term of Proposition \autoref{prop:entrop-estim}.
More precisely, we show the following:

\begin{prop}
\label{prop:Gibbs}
For some universal constant $C$ we have
\begin{align*}
\int_E - \log Leb(\mathcal{J}^E_x \cap \mathcal{Q}_{q,x}^E \cap E \cap A_n) d Leb_{A_n}(x)
&\geq \int_E \phi_g^E(x) \: d Leb_{A_n}(x) - Leb_{A_n}(E) \left ( \frac{\# E }{q} + \# \partial E \log \left ( \frac{C}{\varepsilon} \right ) \right )
\end{align*}
\end{prop}

Fix $x \in A_n \cap E$. Let us note $b_0 = 0$ and
$$
\gls*{E} = \bigcup\limits_{j = 1}^d [\![ a_j ; b_j [\![
$$
We will estimate the Lebesgue measure of the following set:
$$
\gls*{R} := \mathcal{J}^E_x \cap \mathcal{Q}_{q,x}^{E} \cap E \cap A_n
$$
The strategy is to cover $R$ by small sets where the distortion is bounded, then to use the change of variable formula to estimate the Lebesgue measure of each of these sets.

\begin{lem}[Change of variable formula]
\label{lem:chgt-var}
Let $k \in \bbZ_0^+$. If $J$ is a monotone branch of $g^k$ and if $A,B \subset I$, then
$$
Leb(J \cap A \cap g^{-k} B) \leq \frac{1}{\inf\limits_{J \cap A} |(g^k)'|} Leb(B)
$$
\end{lem}
\begin{proof}
We have
\begin{align*}
Leb(B) \geq Leb(g^k(J \cap A \cap g^{-k} B))
\underset{J \text{ monotone branch}}&{=} \int_{A \: \cap \: J \: \cap \: g^{-k} B} |(g^k)'(x)| \: dx\\
&\geq \left ( \inf\limits_{A \: \cap \: J \: \cap \: g^{-k} B} |(g^k)'| \right ) Leb(A \cap J \cap g^{-k} B)\\
&\geq \left ( \inf\limits_{A \: \cap \: J} |(g^k)'| \right ) Leb(A \cap J \cap g^{-k} B)
\end{align*}
\end{proof}

The $g^k$'s that we will consider for the change of variable are the iterates of $g$ at the times of $\partial E$.
Such iterates are of two different types, the ones coming from intervals of times of $E$, of the form $[\![ a_j ; b_j [\![$, and others from intervals of $[\![ 0 ; b_d [\![ \backslash E$.
Note that for a point in $R$, we know to which atoms of $\mathcal{J}$ and $\mathcal{Q}_q$ the points of its orbit belong.
So at these times, the only factor that will come out of the variable change formula is $\phi_g^E(x)$, with some error term.
The reason why the change of variable formula directly gives $\phi_g^E$ at point $x$, and not some integral of $\phi_g^E$, is the following lemma:

\begin{lem}
\label{lem:Qq-util-gibbs}
Let $k \in \bbZ_0^+$. If $y$ and $z$ are in the same atom of $\mathcal{Q}_{q}^k$, then
$$
\left | \log |(g^k)'(y)| - \log |(g^k)'(z)| \right | \leq \frac{k}{q}
$$
\end{lem}

\begin{proof}
Recall that the $I_{q,k}$'s from section \hyperref[ssec:partition]{\textbf{6.b.}} have a length of $1/q$.
\end{proof}

However, for times of $[\![ 0 ; b_d [\![ \backslash E$, the set $R$ may intersect many monotone branches.
But we can only apply the change of variable formula to one monotone branch, so we will have to sum what the change of variable formula gives over all monotone branches.
More precisely, we will define a new partition $\mathcal{V}$ such that we can apply the change of variable formula to each of its atoms, then take the trace of $\mathcal{V}$ over $R$.
First let
$\gls*{Boules}$ be a partition of $I$ whose atoms are intervals of size between $\varepsilon / 27$ and $2 \varepsilon / 28$.
Then, for a partition $\mathcal{R}$ and $n \in \bbZ_0^+$, recall that
$
\mathcal{R}^{\{ n \}} = g^{-n} \mathcal{R}$ and 
$\mathcal{R}^n = \bigvee\limits_{i=0}^{n-1} g^{-i} \mathcal{R}$.
For $j \in [\![ 0 ; d-1 ]\!]$, we define
$$
\gls*{Vaj} = \mathcal{B}^{\{a_{j+1} - b_j \}} \vee \mathcal{J}^{a_{j+1}-b_j} 
$$
We mention that $\mathcal{V}_{a_{j+1}}$ is not a measurable partition of $I$ for $Leb$, since $(g^{a_{j+1}-b_j})'$ may be zero over a set of positive Lebesgue measure.
However, $g^{-b_j} \mathcal{V}_{a_{j+1}}$ does cover $R$ with disjoint atoms, which will be sufficient since we only want to estimate the Lebesgue measure of $R$.\\

To get some properties of the atoms of $\mathcal{V}_{a_{j+1}}$, we will once again use the reparametrizations obtained through the \hyperref[lem:RL]{Reparametrization Lemma} and fixed at the beginning of section \hyperref[sec:mes-emp]{\textbf{5.}} for the definitions, recall that we noted them $\theta_{i^n}$.
Our choice of definition for what a reparametrization and a monotone branch are (sections \hyperref[ssec:reparam-bornee]{\textbf{2.a.}} and \hyperref[ssec:diff-int-cercle]{\textbf{2.d.}}) gives the following lemma:

\begin{lem}
\label{lem:repar-monot}
For $k,l \in \bbZ_0^+$, the interval $(g^k \circ \sigma \circ \theta_{i^{k+l}})_*$ is contained in the closure of a monotone branch of $g^l$.
\end{lem}

\begin{proof}
By item $1)$ of the \hyperref[lem:RL]{Reparametrization Lemma}, the map $g^k \, \circ \, \sigma \, \circ \, \theta_{i^{k+l}}$ is an $(l, \varepsilon)$-bounded reparametrization for $g$.
In particular, for any $i \in [\![ 0 ; l-1 ]\!]$, the derivative of $g \circ \left ( g^{i+k} \circ \sigma \circ \theta_{i^{k+l}} \right )$ does not vanish, hence the derivative of $g$ does not vanish on $\left ( g^{i+k} \circ \sigma \circ \theta_{i^{k+l}} \right )_*$.
Furthermore, the interior $int \left (g \circ \left ( g^{i+k} \circ \sigma \circ \theta_{i^{k+l}} \right ) \right)_*$ does not intersect the point $0$ in $I$.
Then, the set $int \left ( (g^{i+k} \circ \sigma \circ \theta_{i^{k+l}})_* \right )$ is an interval, so it is inside an atom of $\mathcal{J}$.
Therefore, the interval $(g^k \circ \sigma \circ \theta_{i^{k+l}})_*$ is contained in the closure of an atom of $\mathcal{J}^l$.
\end{proof}

We now show that the atoms of $\mathcal{V}_{a_{j+1}}$ have bounded distortion and a size bounded from below.

\begin{lem}
\label{lem:prop-atomes}
Let $j \in [\![ 0 ; d-1 ]\!]$ and $V$ be an atom of $\mathcal{V}_{a_{j+1}}$ such that $g^{-b_j} V$ intersects $R$.
\begin{itemize}
\item[i)] For $y,z \in V$, we have $|(g^{a_{j+1} - b_j})'(y)| \leq \frac{9}{4} |(g^{a_{j+1} - b_j})'(z)|$
\item[ii)] The atom $V$ satisfies
$$
Leb(g^{a_{j+1}-b_j}(V)) \geq \varepsilon / 27
$$
\end{itemize}
\end{lem}

\begin{proof}
Let $y_0 \in R$ be such that $g^{b_j} y_0 \in V$. Before proving the two items, we make some preliminary construction summarized in Figure 4 and detailed on the next page:
\vspace{0em}
\begin{figure}[h!]
\hspace{-2em}\includegraphics[width=47em]{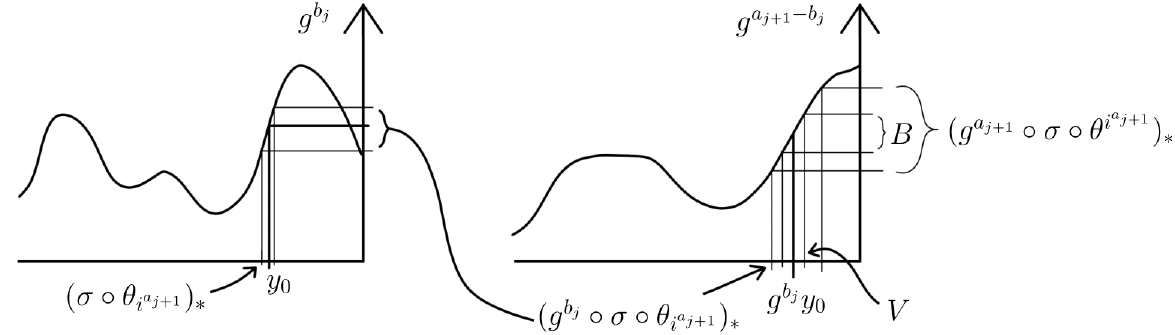}
\label{Figure 4}
\vspace{0em}
\caption{We prove that $y_0$ is in some $\left ( \sigma \circ \theta_{i^{a_{j+1}}} \right )_*$ which when iterated by $g^{a_{j+1}}$ contains $B = \mathcal{B}_{g^{a_{j+1}} y_0}$,\\
so it had to contain $V$ when iterated by only $g^{b_j}$}
\end{figure}

We have $y_0 \in E$, so $a_{j+1} \in E(y_0)$ and there exists a vertex $i^{a_{j+1}} \in \overline{\mathcal{T}_{a_{j+1}}}$ such that $y_0 \in \sigma \circ \theta_{i^{a_{j+1}}}\left ( \left [ - \frac{1}{3} ;\frac{1}{3} \right ] \right )$.
Thus, for $t \in [- 1 ; 1]$, items $1)$ and $3)$ from the \hyperref[lem:RL]{Reparametrization Lemma} give
$$
(g^{a_{j+1}} \circ \sigma \circ \theta_{i^{a_{j+1}}})'(t) \geq \frac{2}{3} \frac{\varepsilon}{6}
$$
Therefore, the image of an interval of length $2/3$ under $g^{a_{j+1}} \circ \sigma \circ \theta_{i^{a_{j+1}}}$ has length at least $\frac{2}{3} \times \frac{\varepsilon}{9}$.
Thus $(g^{a_{j+1}} \circ \sigma \circ \theta_{i^{a_{j+1}}})_*$ contains the ball centered at $g^{a_{j+1}} y_0$ and of radius $\frac{2 \varepsilon}{27}$.
Since we chose atoms of $\mathcal{B}$ to have a length less than $\frac{2 \varepsilon}{28}$, we get $g^{a_{j+1} - b_j}(V) \subset \mathcal{B}_{g^{a_{j+1}} y_0} \subset int \left ( g^{a_{j+1}} \circ \sigma \circ \theta_{i^{a_{j+1}}} \right)_*$, where $int$ denotes the interior of a set.
Then, by definition, the atom $V$ is contained in a monotone branch of $g^{a_{j+1} - b_j}$, but $int \left ( g^{b_j} \circ \sigma \circ \theta_{i^{a_{j+1}}} \right)_*$ as well because of Lemma \autoref{lem:repar-monot}.
Since they intersect, these two branches must be the same.
This shows that $V \subset \left ( g^{b_j} \circ \sigma \circ \theta_{i^{a_{j+1}}} \right)_*$.\\

We now show the two items of the Lemma.
From what precedes, we have $s,t \in [-1 ; 1]$ such that $y = g^{b_j} \circ \sigma \circ \theta_{i^{a_{j+1}}}(t)$ and $z = g^{b_j} \circ \sigma \circ \theta_{i^{a_{j+1}}}(s)$.
Therefore, we have
\begin{align*}
|(g^{a_{j+1} - b_j})'(y)| &= |(g^{a_{j+1} - b_j})' \left ( g^{b_j} \circ \sigma \circ \theta_{i^{a_{j+1}}}(t) \right )|\\
&= \frac{|(g^{a_{j+1}} \circ \sigma \circ \theta_{i^{a_{j+1}}})'(t) )|}{|(g^{b_j} \circ \sigma \circ \theta_{i^{a_{j+1}}})'(t)|}\\
&\leq \frac{\frac{3}{2}|(g^{a_{j+1}} \circ \sigma \circ \theta_{i^{a_{j+1}}})'(s) )|}{\frac{2}{3}|(g^{b_j} \circ \sigma \circ \theta_{i^{a_{j+1}}})'(s)|}\\
&= \frac{9}{4} |(g^{a_{j+1} - b_j})'(z)|
\end{align*}
For the second item, we will in fact prove that
$$
B := \mathcal{B}_{g^{a_{j+1}} y_0} = g^{a_{j+1 - b_j}}(V)
$$
We proceed by double inclusion.
By definition of $V$, we have $V = (g^{-(a_{j+1} - b_j)} B) \cap J$ where $J = \mathcal{J}_{g^{b_j}y_0}^{a_{j+1} - b_j}$.
So we are left to show that $B \subset g^{a_{j+1 - b_j}}(V)$.
From the preliminary construction, we know that $int \left ( g^{b_j} \circ \sigma \circ \theta_{i^{a_{j+1}}}\right)_* \subset J$ and $B \subset int \left ( g^{a_{j+1}} \circ \sigma \circ \theta_{i^{a_{j+1}}}\right)_*$.
Therefore, for any $y \in B$, there exists a $t \in ]-1 ; 1[$ such that $y=g^{a_{j+1}} \circ \sigma \circ \theta_{i^{a_{j+1}}}(t)$, hence $y' := g^{b_j} \circ \sigma \circ \theta_{i^{a_{j+1}}}(t)$ satisfies
$$
y = g^{a_{j+1}-b_j}(y') \; \; \; \text{and} \; \; \; y' \in (g^{-(a_{j+1} - b_j)} B) \cap J = V
$$
\end{proof}

We define $\mathcal{V}$, a partition of $R = \mathcal{J}^E_x \cap \mathcal{Q}_{q,x}^{E} \cap E \cap A_n$, as follows
$$
\gls*{V} = \bigvee\limits_{j=0}^{d-1} g^{-b_j}\mathcal{V}_{a_{j+1}}
$$
In the following lemma, we estimate the Lebesgue measure of $R$ over an atom of $\mathcal{V}$ by using the change of variable formula.

\begin{lem}
\label{lem:leb-atome}
For any atom $V = \bigcap\limits_{j = 0}^{d-1} g^{-b_j} V_{a_{j+1}}$ of $\mathcal{V}$, we have
$$
Leb(R \cap V) \leq (C / \varepsilon)^{\# \partial E} e^{- \phi_g^E(x) + \# E  / q} \left ( \prod\limits_{j = 0}^{d-1} Leb(V_{a_{j+1}}) \right )
$$
\end{lem}

\begin{proof}
Since $R$ forces a precise monotone branch during times of $E$, and because $V$ also forces one during times of $[\![ 0 ; b_d [\![ \backslash E$, we know that the set $R \cap V$ is contained in a monotone branch of $g^{b_d}$. Therefore Lemma \autoref{lem:chgt-var} gives
$$
Leb(R \cap V) \leq \left ( \inf\limits_{R \cap V} |(g^{b_d})'|\right )^{-1} Leb(g^{b_d}(R \cap V))
$$
Then, for $y \in I$, we have
$$
|(g^{b_d})'(y)| = |(g^{a_1})'(y)| \times |(g^{b_1 - a_1})'(g^{a_1} y)| \times ... \times |(g^{b_d - a_d})'(g^{a_d}y)|
$$
Thus, when $y \in R \cap V$, we have
\begin{align*}
|(g^{b_d})'(y)| \geq \left ( \inf\limits_{V_{a_1}} |(g^{a_1})'| \right ) \times \left ( \inf\limits_{\mathcal{Q}_{q,g^{a_1} x}^{b_1 - a_1}} |(g^{b_1-a_1})'| \right ) \times ... \times \left ( \inf\limits_{\mathcal{Q}_{q, g^{a_d} x}^{b_d - a_d}} |(g^{b_d - a_d})'| \right )
\end{align*}
For the factors in $\mathcal{Q}_q$, if $j \in [\![ 1 ; d ]\!]$ and $z \in \mathcal{Q}_{q, g^{a_j} x}^{b_j - a_j}$, Lemma \autoref{lem:Qq-util-gibbs} implies
$$
\log | (g^{b_j-a_j})'(z)| \geq \log | (g^{b_j-a_j})'(g^{a_j} x)| - \frac{b_j - a_j}{q}
$$
Then, for $j \in [\![ 0 ; d-1 ]\!]$, the set $V_{a_{j+1}}$ is inside a monotone branch of $g^{a_{j+1}-b_j}$, so the argument that we used to prove the change of variable formula (Lemma \autoref{lem:chgt-var}) gives
\begin{align*}
Leb(V_{a_{j+1}}) &\geq \left ( \sup\limits_{V_{a_{j+1}}} |(g^{a_{j+1} - b_j})'| \right )^{-1} Leb(g^{a_{j+1} - b_j}(V_{a_{j+1}}))\\
\underset{\text{Lemma } \autoref{lem:prop-atomes}}&{\geq} \frac{4}{9} \left ( \inf\limits_{V_{a_{j+1}}} |(g^{a_{j+1} - b_j})'| \right )^{-1} \frac{\varepsilon}{27}
\end{align*}
Here, Lemma \autoref{lem:prop-atomes}'s hypothesis is satisfied because we may assume that $g^{-b_j} V_{a_{j+1}}$ intersects $R$.
This gives the following estimate for the Lebesgue measure of $R \cap V$:
\begin{align*}
Leb(R \cap V) \leq e^{- \phi_g^E(x)} e^{\# E / q} \left ( \prod\limits_{j=0}^{d-1} \frac{243}{4 \varepsilon} Leb(V_{a_{j+1}}) \right )
\end{align*}
And because $d = \# \partial E / 2$, we may take $C = 8$.
\end{proof}

\begin{proof}[Proof of Proposition \autoref{prop:Gibbs}]
Recall that we fixed $x \in A_n \cap E$ and let $R := \mathcal{J}^E_x \cap \mathcal{Q}_{q,x}^{E} \cap E \cap A_n$.
we write
\begin{align*}
Leb\left(R\right) &= \sum\limits_{V \in \mathcal{V}} Leb\left(R \cap V \right)\\
&\leq \sum\limits_{V_{a_1} \in \mathcal{V}_{a_1}} \sum\limits_{V_{a_2} \in \mathcal{V}_{a_2}} ... \sum\limits_{V_{a_d} \in \mathcal{V}_{a_d}} Leb\left(R \cap \bigcap\limits_{j=0}^{d-1} g^{-b_j} V_{a_{j+1}}\right)\\
\underset{\text{Lemma } \autoref{lem:leb-atome}}&{\leq} \sum\limits_{V_{a_1} \in \mathcal{V}_{a_1}} \sum\limits_{V_{a_2} \in \mathcal{V}_{a_2}} ... \sum\limits_{V_{a_d} \in \mathcal{V}_{a_d}} (C/ \varepsilon)^{\# \partial E} e^{- \phi_g^E(x) + \# E /q} \prod\limits_{j=0}^{d-1} Leb\left(V_{a_{j+1}}\right)\\
&\leq \left(C/ \varepsilon \right)^{\# \partial E} e^{- \phi_g^E(x) + \# E / q}
\end{align*}
We then integrate over $x$ and get
\begin{align*}
\int_E - \log Leb\left(\mathcal{J}^E_x \cap \mathcal{Q}_{q,x}^E \cap E \cap A_n \right) d Leb_{A_n}(x)
&\geq \int_E - \log \left (  (C/ \varepsilon)^{\# \partial E} e^{- \phi_g^E(x) + \# E / q} \right ) d Leb_{A_n}(x)\\
&= \int_E \phi_g^E(x) d Leb_{A_n}(x) - Leb_{A_n}(E) \left ( \frac{\# E }{q} + \# \partial E \log \left ( \frac{C}{\varepsilon} \right ) \right )
\end{align*}
\end{proof}

\section{Proof of Theorem 1}
\label{sec:demoTh}

\subsection{Absolute continuity}
\label{ssec:SRB}

To get an absolutely continuous measure, we will use the following entropy characterization:

\begin{thm}[Theorem VII.1.1 from \cite{SETE}]
\label{th:form-entropie}
Let $f : I \to I$ be a $\mathcal{C}^r$ map where $r > 1$ and $I$ is the interval $[0;1]$ or the circle $\bbT^1$.
Let $\mu$ be an $f$-invariant hyperbolic borelian probabily satisfying the integrability condition
$$
\log | f' | \in \mathbb{L}^1(I,\mu)
$$
Then $\mu$ is absolutely continuous with respect to the Lebesgue measure on $I$ if and only if it satisfies the entropy formula
$$
h_{\mu}(f) = \int_I \chi_f \: d\mu
$$
\end{thm}

In the previous sections, we did not work directly with $f$ but rather with $g$, an iterate of $f$. Let us first show how to obtain a hyperbolic absolutely continuous invariant borelian probability (
\gls*{HACIP}) for $f$ once we have one for $g$.
Suppose $\mu$ is an HACIP for $g = f^p$, and consider the following $f$-invariant borelian probability
$$
\gls*{Nu} = \frac{1}{p}\sum\limits_{k=0}^{p-1} f^k_* \mu
$$
One can verify that $f$ has a positive Lyapunov exponent $\nu$-almost everywhere and that $\nu$ satisfies the integrability condition.
For the entropy formula, we have
$$
p h_f(\nu) = h_{f^p}(\mu) = \chi_{f^p}(\mu) = p \chi_f(\nu)
$$
This shows that $\nu$ is an HACIP for $f$. In particular, any ergodic component of $\nu$ is an ergodic HACIP for $f$.\\

We now build a measure satisfying the entropy formula and the integrability condition for $g = f^p$, where $p$ and other notations are defined at the beginning of section \hyperref[sec:mes-emp]{\textbf{5.}}, so we get an integer sequence $\mathcal{n}$ and converging sequences of measures $(\mu_n^{M,m})_{n \in \mathcal{n}, M,m \in \bbZ_0^+}$ and $(\nu_n^{M,m})_{n \in \mathcal{n}, M,m \in \bbZ_0^+}$.
We show that the limit $\mu = \lim\limits_{M \to +\infty} \lim\limits_{\mathcal{n} \ni n \to +\infty} \mu_n^{M,m}$ is an HACIP.
Item $ii)$ from Proposition \autoref{prop:CV-mes} implies that the limit $\mu$ is a $g$-invariant borelian probability.
The integrability condition and the positivity of the Lyapunov exponent $\mu$-almost everywhere come from Propositions \autoref{prop:phi-integ} and \autoref{prop:expo-pos}.
Hence, we are only left to show that $\mu$ satisfies the entropy formula.

\begin{prop}
\label{prop:formule-entropie}
The measure $\mu$ satisfies the entropy formula for $g$.
\end{prop}

\begin{proof}
By Ruelle's inequality \cite{Ruelle1978} and Proposition \autoref{prop:expo-pos}, we only have to show that $h_{\mu}(g) \geq \phi_g(\mu)$.
For $p \in \bbZ_0^+$, item $1)$ from Proposition \autoref{prop:partition} gives that the $\mu$-partition $\mathcal{P}_q$ is of finite $\mu$-entropy.
We can thus write
$$
h_{\mu}(g) \geq h_{\mu}(g,\mathcal{P}_q) = \lim\limits_{m \to +\infty} \frac{1}{m} H_{\mu}(\mathcal{P}_q^m)
$$

\hspace{1em}\textbf{Part 1)} Dealing with the error terms:\\

Let $m,q \in \bbZ^+$, $M \in \bbZ_0^+$ and $n \in \mathcal{n}$.
The purpose of the previous sections was to reach an inequality of the form
$$
\frac{1}{m} H_{\mu_n^{M,m}}(\mathcal{P}_q^m) \geq \int \phi_g d \mu_n^{M,m} + \text{ error term}
$$
In this first part, we prove that this error term goes to 0 when $n$ then $M$ then $q$ go to infinity. We will prove in part \textbf{2)} that
$\lim\limits_{M \to +\infty} \lim\limits_{\mathcal{n} \ni n \to +\infty} H_{\mu_n^{M,m}}(\mathcal{P}_q^m) = H_{\mu}(\mathcal{P}_q^m)$ and in part \textbf{3)} that $\lim\limits_{M \to +\infty} \lim\limits_{\mathcal{n} \ni n \to +\infty} \phi_g(\mu_n^{M,m}) = \phi_g(\mu)$.\\

In section \hyperref[ssec:entropie1]{\textbf{6.a.}}, we proved Proposition \autoref{prop:entrop-estim} claiming that
\begin{align*}
\frac{1}{m} \left ( \int \# E_n^{M,m}(x) d Leb_{A_n}(x) \right ) H_{\mu_n^{M,m}}(\mathcal{P}_q^m)
&\geq \sum\limits_{E \in \mathcal{E}_n^{M,m}} \int_E - \log Leb(\mathcal{P}_{q,x}^E \cap E \cap A_n) d Leb_{A_n}(x)\\
&\hspace{3em}+ Leb_{A_n}(E) \log Leb_{A_n}(E)\\
&\hspace{3em}+ Leb_{A_n}(E) \log Leb(A_n)\\
&\hspace{1em}- m \sum\limits_{E \in \mathcal{E}_n^{M,m}} Leb_{A_n}(E) \log ( \# (\mathcal{P}_q)_{\mu^E}) \# \partial E
\end{align*}
To show that these first error terms go to $0$ as $n$ then $M$ go to infinity, we use items $i)$ and $iv)$ from Lemma \autoref{lem:def-An} and item $4)$ from Proposition \autoref{prop:partition}, which imply the following
$$
\limsup\limits_{\mathcal{n} \ni n \to +\infty} -\frac{1}{n}\sum\limits_{E \in \mathcal{E}_n^{M,m}} Leb_{A_n}(E) \log Leb_{A_n}(E) + Leb_{A_n}(E) \log Leb(A_n) - m Leb_{A_n}(E) \# \partial E \log \# ((\mathcal{P}_q)_{\mu^E}) \underset{M \to +\infty}{\longrightarrow} 0
$$
Then, for the main term, the Gibbs inequality (Proposition \autoref{prop:Gibbs}) gave us
\begin{align*}
\sum\limits_{E \in \mathcal{E}_n^{M,m}} \int_E - \log Leb(\mathcal{J}^E_x \cap \mathcal{Q}_{q,x}^E \cap E \cap A_n) d Leb_{A_n}(x)
&\geq \int \phi_g^{E_n^{M,m}(x)}(x) \: d Leb_{A_n}(x) - \frac{1}{q} \int \# E_n^{M,m}(x) d Leb_{A_n}(x)\\
&\hspace{1em}- \log \left ( \frac{C}{\varepsilon} \right ) \int \# \partial E_n^{M,m}(x) d Leb_{A_n}(x)
\end{align*}
To show that these other error terms converge to $0$ as well, we use item $ii)$ from Lemma \autoref{lem:EnM} and let $q$ go to infinity. Therefore, we obtain
$$
\frac{1}{m} H_{\mu_n^{M,m}}(\mathcal{P}_q^m)
\geq \phi_g(\mu_n^{M,m}) - \varepsilon_{n,M,m,q}
$$
where $\lim\limits_{M \to +\infty} \lim\limits_{\mathcal{n} \ni n \to +\infty}\varepsilon_{n,M,m,q} = \varepsilon_q \underset{q \to +\infty}{\longrightarrow} 0$.
Hence, parts \textbf{2)} and \textbf{3)} of the proof will give
$$
\forall q \in \bbZ^+, h_{\mu}(g) \geq \phi_g(\mu) - \varepsilon_q
$$
Hence letting $q \to +\infty$ will give the result.\\

\hspace{1em} \textbf{Part 2)} Convergence of the entropy:\\

We now show that $\lim\limits_{M \to +\infty} \lim\limits_{\mathcal{n} \ni n \to +\infty} H_{\mu_n^{M,m}}(\mathcal{P}_q^{m}) = H_{\mu}(\mathcal{P}_q^m)$.
In fact, we will rather lead the computations for $\nu_n^{M,m}$, which will imply the result for $\mu_n^{M,m}$.\\

We first prove that 
$$
\# \{ P \in \mathcal{P}_q^{m} \mid \exists n \in \mathcal{n}, \nu_n^{M,m}(P) > 0 \} < +\infty
$$
For $P$ in this set, we write $P = \bigcap\limits_{j = 0}^{m-1} g^{-j} P_j$.
Then, there exists an $n \in \mathcal{n}$, $x \in A_n$ and $i \in E_n^{M,m}(x)$ such that for any $j \in [\![ 0 ; m-1 ]\!]$, we have $g^{i+j} x \in P_j$.
By definition of $E_n^{M,m}(x)$, we have $i+j \in E_n^M(x)$, and item $iii)$ from Lemma \autoref{lem:tpsHB} gives
$|g'(g^{i+j} x)| \geq || g' ||_{\infty}^{-M}$ .
Therefore, the same reasoning as for item $2)$ of Proposition \autoref{prop:partition} gives that there are only finitely many such $P_j$, hence finitely many such $P$.
Then, the choice of $a$ done at the beginning of section \hyperref[ssec:partition]{\textbf{6.b.}} gives $\nu^{M,m}(\partial \mathcal{P}_q) = 0$, therefore
$$
H_{\nu_n^{M,m}}(\mathcal{P}_q^{m}) \underset{\mathcal{n} \ni n \to +\infty}{\longrightarrow} H_{\nu^{M,m}}(\mathcal{P}_q^m)
$$
For the limit in $M$, note $\psi : x \mapsto - x \log(x)$, so that
\begin{align*}
H_{\nu^{M,m}}(\mathcal{P}_q^{m})
= \sum\limits_{P \in \mathcal{P}_q^{m}} \psi(\nu^{M,m}(P))
\end{align*}
For $P \in \mathcal{P}_q^m$, item $iii)$ from Proposition \autoref{prop:CV-mes} gives that $\nu^{M,m}(P) \underset{M \to +\infty}{\nearrow} \mu(P)$.
Then notice that $\psi$ is continuous, non-negative and non-decreasing on $[0 ; 1/e]$, and that at most two atoms of $\mathcal{P}_q$ will ever have a $\nu^{M,m}$-measure larger than $1/e$.
Therefore, using monotone convergence, we get
$$
\sum\limits_{P \in \mathcal{P}_q^{m}} \psi(\nu^{M,m}(P)) \underset{M \to +\infty}{\longrightarrow} \sum\limits_{P \in \mathcal{P}_q^{m}} \psi(\mu(P))
= H_{\mu}(\mathcal{P}_q^m)
$$
This gives the result for $\nu_n^{M,m}$. To obtain the same for $\mu_n^{M,m}$, let us note
$$
\gls*{BetanMm} = \int d_n(E_n^{M,m}(x)) \: d Leb_{A_n}(x)
$$
By item $ii)$ from Lemma \autoref{lem:def-An}, this quantity satisfies
$$
\beta_n^{M,m} \underset{\mathcal{n} \ni n \to +\infty}{\longrightarrow} \beta_m^M \underset{M \to +\infty}{\longrightarrow} \beta^{\infty} \; \; \; \text{and} \; \; \; \mu_n^{M,m} = \frac{\beta^{\infty}}{\beta_n^{M,m}} \nu_n^{M,m}
$$
Hence
\begin{align*}
H_{\mu_n^{M,m}}(\mathcal{P}_q^m)
&= \sum\limits_{P \in \mathcal{P}_q^m} - \mu_n^{M,m}(P) \log(\mu_n^{M,m}(P))\\
&= - \log \frac{\beta^{\infty}}{\beta_n^{M,m}} + \sum\limits_{P \in \mathcal{P}_q^m} - \frac{\beta^{\infty}}{\beta_n^{M,m}} \nu_n^{M,m}(P) \log(\nu_n^{M,m}(P))\\
&= - \log \frac{\beta^{\infty}}{\beta_n^{M,m}} + \frac{\beta^{\infty}}{\beta_n^{M,m}} H_{\nu_n^{M,m}}(\mathcal{P}_q^m)
\underset{n,M}{\longrightarrow} H_{\mu}(\mathcal{P}_q^m)
\end{align*}

\hspace{1em}\textbf{Part 3)} Convergence of $\phi_g(\mu_n^{M,m})$\\

We show that $\lim\limits_{M \to +\infty} \lim\limits_{\mathcal{n} \ni n \to +\infty} \phi_g(\mu_n^{M,m}) = \phi_g(\mu)$.
For the limit in $n$, we use the fact that $\phi_g$ is continuous everywhere except on critical points.
However, item $iii)$ from Lemma \autoref{lem:tpsHB} gives
$$
\forall M,m \in \bbZ_0^+, \forall n \in \mathcal{n}, \forall x \in \sigma_*, \forall k \in E_n^{M,m}(x), \phi_g(g^k x) \geq -M \log || g' ||_{\infty}
$$
This shows that the support of $\mu_n^{M,m}$ is at distance at least $\varepsilon_M$ from the set $\mathcal{C}_g$, where $\varepsilon_M>0$ does not depend on $n$ nor $m$.
Therefore, we have
$$
\forall M,m \in \bbZ_0^+, \phi_g(\mu_n^{M,m}) \underset{\mathcal{n} \ni n \to \infty}{\longrightarrow}
\phi_g(\mu^{M,m})
$$
Notice that the same holds for $\nu_n^{M,m}$, so that $\lim\limits_{\mathcal{n} \ni n \to +\infty}\phi_g(\nu_n^{M,m}) = \phi_g(\nu^{M,m})$, for any $M,m \in \bbZ_0^+$.
We now prove that $\left ( \phi_g(\nu^{M,m}) \right )_M$ converges to $\phi_g(\mu)$.
Notice that this is true if $\phi_g$ were to be a characteristic function, because of item $iii)$ from Proposition \autoref{prop:CV-mes}.
By linearity, monotone convergence, and item $iii)$ from Proposition \autoref{prop:CV-mes}, which gives that $\left ( \nu^{M,m} \right )_M$ is non-decreasing, it is true for any non-negative measurable function.
This is therefore true for any function that is $\mu$-integrable and $\nu^{M,m}$-integrable for every $M,m \in \bbZ_0^+$. We thus get the convergence of $\left ( \phi_g(\nu^{M,m})\right )_M$ from Proposition \autoref{prop:phi-integ}. To prove the convergence of $\left ( \phi_g(\mu^{M,m})\right )_M$, we write
\begin{align*}
\phi_g(\mu^{M,m}) &= \lim\limits_{n \to \infty} \phi_g(\mu_n^{M,m})\\
&= \lim\limits_{n \to \infty} \frac{n \beta^{\infty}}{\int \# E_n^{M,m}(x) \: d Leb_{A_n}(x)} \phi_g(\nu_n^{M,m})\\
\text{\hspace{-8em}From Lemma \autoref{lem:def-An}.}ii) \text{ and previous remarks} \; \; \; &\underset{M \to +\infty}{\longrightarrow} \phi_g(\mu)
\end{align*}
\end{proof}

\subsection{Cover by the basins}
\label{ssec:bassins}

\begin{prop}
\label{prop:rec-bassins}
We have the following inclusion Lebesgue-almost everywhere
$$
\left \{ \chi > R(f)/r \right \} \subset \bigcup\limits_{\mu \text{ HACIP ergodic}} \mathcal{B}(\mu) \cap \{ \chi = \chi(\mu) \}
$$
\end{prop}

\begin{proof}
Let 
$\mathcal{B} = \bigcup\limits_{\mu \text{ HACIP ergodic}} \mathcal{B}(\mu) \cap \{ \chi = \chi(\mu) \}$.
By contradiction, assume that there is $A \subset \{ \chi > R(f)/r \}$ such that $Leb(A) > 0$ and $A \cap \mathcal{B} = \emptyset$.
We may take a very small subset of $A$, still of positive Lebesgue measure, such that every $x \in A$ has the same set of geometric times $E(x)$, note it $E_0$, and note $\beta > 0$ the lower bound of the upper density of $E_0$.
Let $A'$ be the set of density points of $A$.
Let $x \in A$ and $n \in E_0$, note $\theta_{i^n} \in \overline{\mathcal{T}_n}$ a reparametrization such that $x \in \sigma \circ \theta_{i^n} \left ( \left [ - \frac{1}{3} ; \frac{1}{3} \right ] \right)$, and note
$$
H_n(x) = \left ( \sigma \circ \theta_{i^n} \right )_* \; \; \; \text{and} \; \; \; D_n(x) = g^n H_n(x) = \left ( g^n \circ \sigma \circ \theta_{i^n} \right )_*
$$
Therefore, $D_n(x)$ is an interval of length at least $\frac{\varepsilon}{27}$ on each side of $g^n(x)$.
We also have
$$
\forall x \in A', \frac{Leb(H_n(x) \cap A)}{Leb(H_n(x))} \underset{E_0 \ni n \to +\infty}{\longrightarrow} 1
$$
Let $A''$ be a subset of $A'$, still of positive Lebesgue measure, such that this convergence is uniform on $A''$.
We now build an HACIP $\mu$ as in the previous sections but starting from $A''$. We keep the same notations as in section \hyperref[sec:mes-emp]{\textbf{5.}}, that is $g=f^p$ and $\nu = \frac{1}{p} \sum\limits_{k=0}^{p-1} f^k_* \mu$ is an HACIP for $f$.
Thus every ergodic component $\nu_{erg}$ of $\nu$ satisfies $\nu_{erg}(\mathcal{B}) = 1$, so $\nu(\mathcal{B})=1$, and $\mu(\mathcal{B}) = 1$.
Then we notice that
$$
\{ x \in I \mid \mu(]x - \frac{\varepsilon}{54} ; x + \frac{\varepsilon}{54}[) = 0 \} \subset I \backslash \text{supp } \mu
$$
So by using $\mu \ll Leb$, we get
$$
\mu \left ( \{ x \in I \mid Leb(]x - \frac{\varepsilon}{54} ; x + \frac{\varepsilon}{54}[ \cap \mathcal{B}) = 0 \} \right ) = 0
$$
Therefore, there exists $c > 0$ such that
$$
\mu \left ( \{ x \in I \mid Leb(]x - \frac{\varepsilon}{54} ; x + \frac{\varepsilon}{54}[ \cap \mathcal{B}) > c \} \right ) > 1 - \beta
$$
Let $G$ be the $\frac{\varepsilon}{54}$-open neighborhood of the set $\{ x \in I \mid Leb(]x - \frac{\varepsilon}{54} ; x + \frac{\varepsilon}{54}[ \cap \mathcal{B}) > c \}$.
For $n \in \mathcal{n}$, note
$$
\zeta_n = \int \frac{1}{n} \sum\limits_{k \in E_0 \cap [\![ 0 ; n [\![} \delta_{g^k x} d Leb_{A_n}(x)
$$
Let $\zeta$ be an accumulation point of $(\zeta_n)_{n \in \mathcal{n}}$ for the weak-$*$ topology.
Since $\zeta_n(I) \geq \beta$, we have $\zeta(I) \geq \beta$.
Also notice that $\zeta \leq \mu$, hence
\begin{align*}
\zeta(G) &= \mu(G) + \zeta(G) - \mu(G)\\
&\geq \mu(G) + \zeta(I) - \mu(I)\\
&> 1- \beta + \beta - 1 = 0
\end{align*}
Thus $0 < \zeta(G) \leq \liminf\limits_{\mathcal{n} \ni n \to +\infty} \zeta_n(G)$, and there exist infinitely many $n$ and $x_n \in A''$ such that
$$
g^n x_n \in G \; \; \;\text{and} \; \; \; n \in E_0
$$
Notice that for any $n$, we have $\mathcal{B} \cap g^n A \subset g^n(A \cap \mathcal{B}) = \emptyset$.
Then recall that $D_n(x_n)$ is an interval of length at least $\frac{\varepsilon}{27}$ on each side of $g^n x_n$, for infinitely many $n$ inside $E_0$.
This implies that
\begin{align*}
0 &= Leb(\mathcal{B} \cap g^n A )\\
&\geq Leb(D_n(x_n) \cap \mathcal{B} \cap g^n A)\\
&= Leb(D_n(x_n) \cap \mathcal{B}) - Leb(D_n(x_n) \cap \mathcal{B} \backslash g^n A)\\
&\geq c - Leb(D_n(x_n) \backslash g^n A)
\end{align*}
Our goal is to show that $Leb(D_n(x_n) \backslash g^n A) \to 0$.
We use the fact that $g^n$ has bounded distortion on $H_n(x_n)$, which gives
\begin{align*}
Leb(D_n(x_n) \backslash g^n A) &= Leb(D_n(x_n)) \frac{Leb(D_n(x_n)  \backslash g^n A)}{Leb(D_n(x_n))}\\
&\leq 2 \varepsilon \times \frac{9}{4} \frac{Leb(H_n(x_n) \backslash A)}{Leb(H_n(x_n))}
\end{align*}
We conclude using uniform convergence on $A''$ and the fact that $x_n \in A''$.
\end{proof}

We now prove the finiteness of ergodic absolutely continuous measures whose Lyapunov exponent is larger than $\frac{R(f)}{r} + \delta$, for any $\delta > 0$.

\begin{prop}
\label{prop:fini}
For $\delta > 0$, there are finitely many ergodic HACIP whose basins intersect $\left \{ \chi > \frac{R(f)}{r} + \delta\right\}$ with positive Lebesgue measure.
\end{prop}

\begin{proof}
We prove that any such measure must have the Lebesgue measure of its basin bounded from below by a constant depending only on $\delta, r$ and $f$.
Let $b = \frac{R(f)}{r} + \delta$ and apply Proposition \autoref{prop:dens-pos} to this $b$, which gives $p,  \beta$ and $\varepsilon$ that depend only on $\delta, r$ and $f$.
%
%
Denote by $\mathcal{B}_{bounded} = \bigcup\limits_{
\substack{
\mu \text{ HACIP ergodic s.t.}\\
Leb(\mathcal{B}(\mu)) \geq 8 \varepsilon / 243
}} \mathcal{B}(\mu)$.
By contradiction, assume that there exists $A \subset \left \{ \chi > \frac{R(f)}{r} + \delta\right\}$ of positive Lebesgue measure such that $A \cap \mathcal{B}_{bounded} = \emptyset$.
We define $E_0, A', A''$ as in the proof of Proposition \autoref{prop:rec-bassins}, and let $\mu$ be a $g$-invariant measure obtained by using the construction of the previous sections but starting from $A''$, then let $\nu = \frac{1}{p}\sum\limits_{k=0}^{p-1} f^k_* \mu$.
We prove that any ergodic component of $\nu$ must have its basin of Lebesgue measure larger than ${8 \varepsilon}/{243}$.
This is enough to conclude, since it gives $\nu(\mathcal{B}_{bounded}) = 1$, then $\mu(\mathcal{B}_{bounded}) = 1$, and the rest of proof of Proposition \autoref{prop:rec-bassins} gives, for any $n \in E_0$, that 
$$
0 = Leb(\mathcal{B}_{bounded} \cap g^n A) \geq c + \underset{E_0 \ni n \to +\infty}{o(1)}
$$
Let $\nu_{erg}$ be an ergodic component of $\nu$.
Then let $x$ be a density point of $\mathcal{B}(\nu_{erg})$.
Thus, for any $n \in E_0$, the fact that $g^n$ has bounded distorsion on $H_n(x)$ gives
\begin{align*}
Leb(\mathcal{B}(\nu_{erg})) &\geq Leb\left ( \mathcal{B}(\nu_{erg}) \cap D_n(x) \right )\\
&= Leb ( D_n(x)) \frac{Leb\left ( \mathcal{B}(\nu_{erg}) \cap D_n(x) \right )}{Leb ( D_n(x))}\\
&\geq Leb ( D_n(x)) \frac{4}{9} \frac{Leb \left ( \mathcal{B}(\nu_{erg}) \cap H_n(x) \right )}{Leb(H_n(x))}\\
&\geq \frac{2 \varepsilon}{27} \times \frac{4}{9} \left ( 1 + \underset{E_0 \ni n\to +\infty}{o(1)} \right )
\end{align*}
Hence, letting $n \in E_0$ go to infinity concludes.
\end{proof}

We then explain how to obtain the bound stated in Proposition \autoref{prop:title}.

\begin{proof}[Proof of Proposition \autoref{prop:title}]
By applying Proposition  \autoref{prop:dens-pos} to $b = \frac{\log || f' ||_{\infty}}{r} + \delta \geq \frac{R(f)}{r} + \delta$, we obtain an $\varepsilon > 0$ and a $p$ large enough such that the set of geometric times $E_p(x)$ has positive density for Lebesgue-almost every $x$.
Then, in Proposition \autoref{prop:fini}, we showed that basins to consider are of Lebesgue measure larger than $\frac{8 \varepsilon}{243}$.
Therefore, to estimate the number of these basins, we estimate $\varepsilon$.
From the proof of the \hyperref[lem:RL]{Reparametrization Lemma}, we can take $\varepsilon$ such that $(2 \varepsilon)^{\min(2,r) - 1} < \frac{1}{2 || (f^p)' ||_{r-1}}$.
Thus, we first estimate $p$.
From the proof of Proposition \autoref{prop:dens-pos}, it suffices to choose $p$ such that
$$
-b + \frac{\log(\log || (f^p)' ||_{\infty} + 1 - pb+1)+1}{p} + \frac{\log(2 C_r p \log || f' ||_{\infty})}{p} + \frac{\log || (f^p)' ||_{\infty} + 1 - pb}{p(r-1)} < 0
$$
One can check that this is obtained by taking $p$ such that
$$
\frac{2\log(p B_r \log || f' ||_{\infty})}{p} < \delta
$$
where $B_r$ is a constant depending on $r$.
Thus, we can take $p = \frac{4}{\delta} \log \frac{2 B_r \log || f' ||_{\infty}}{\delta}$.
We now estimate $|| (f^p)' ||_{r-1}$.
From Faà di Bruno's formula \cite{Encinas2003ASP}, we have that
$
|| (f^p)' ||_{r-1} \leq A^r || (f^{p-1})' ||_{r-1} || f' ||_{r-1}^r
$,
where $A$ is a universal constant.
Therefore,
$$
|| (f^p)' ||_{r-1} \leq A^{pr} || f' ||_{r-1}^{pr}
$$
In the end, we obtain that the number of basins is bounded by
\begin{align*}
\frac{243}{8 \varepsilon} &\leq C \left (2 || (f^p)' ||_{r-1} \right )^{\frac{1}{\min(2,r)-1}}\\
&\leq C \left ( A^{pr} || f' ||_{r-1}^{pr} \right )^{\frac{1}{\min(2,r)-1}}\\
&\leq (A || f' ||_{r-1})^{\frac{r}{\min(2,r)-1}\times \frac{4}{\delta} \log \frac{2 B_r \log || f' ||_{\infty}}{\delta}}\\
&= \exp \left ( \frac{r}{\min(2,r)-1}\times \frac{4}{\delta} \log \frac{2 B_r \log || f' ||_{\infty}}{\delta} \times \log (A || f' ||_{r-1}) \right )\\
&\leq \left (\frac{\log || f' ||_{\infty}}{\delta} \right )^{ (C_r \log || f' ||_{r-1} ) / \delta }
\end{align*}
\end{proof}

\begin{rmq}
It is possible to estimate the constant of the \hyperref[lem:RL]{Reparametrization Lemma} and obtain that it is smaller than $C r^{2r}$, where $C$ is a universal constant. This implies that the constant $C_r$ of Proposition \autoref{prop:title} is smaller that $C r \sqrt{r} \log(r)$. Then, we can show that if $f$ is $\mathcal{C}^{\infty}$, then there exists a constant $C$ depending on $|| f' ||_{\infty}$ such that, for any $\delta > 0$, the number of hyperbolic ergodic $f$-invariant absolutely continuous measures whose basin intersects $\{ \chi > \delta \}$ with positive Lebesgue measure is less than
$
\left ( || f' ||_{C/\delta} \right )^{C / \delta^3}
$.
\end{rmq}

\begin{rmq}
\label{rmq:analytic}
If $f$ is analytic, we have the following fact:
$
\exists \varepsilon > 0, \forall k \in \bbZ^+, || d^k f||_{\infty} \leq \left ( \frac{k}{\varepsilon} \right )^k
$.
Consequently, there exists a constant $C$ depending only on $|| f' ||_{\infty}$ such that, for $\delta < \varepsilon$, the number of hyperbolic ergodic $f$-invariant absolutely continuous measures whose basin intersects $\{ \chi > \delta \}$ with positive Lebesgue measure is less than
$
C^{1 / \delta^4}
$. 
\end{rmq}

\begin{rmq}
\label{rmq:trans}
We now explain the transitive case. Suppose that $\mu$ is an ergodic HACIP and note $\rho$ its probability density with respect to the Lebesgue measure.
By using Lemma VII.9.1 from \cite{SETE}, we obtain that there exists an open set where $\rho$ is positive.
If we do the same for another ergodic HACIP $\nu$, then we get an open set $V$ where the probability density of $\nu$ is positive.
By applying transitivity to $U$ and $V$, one can show that $Leb(\mathcal{B}(\nu) \cap U) > 0$, which implies that $\mu(\mathcal{B}(\nu)) > 0$, so $\mu=\nu$.
When $\mu$ and $\nu$ are not supposed ergodic, then the previous case shows that all the ergodic components of $\mu$ and $\nu$ are equal, so $\mu$ and $\nu$ are in fact ergodic and $\mu=\nu$.
\end{rmq}

%
%
%
%
%

\vspace{2em}
\phantomsection
\addcontentsline{toc}{section}{References}


%

\vspace{2em}

\phantomsection
\addcontentsline{toc}{section}{Index of notations}

\section*{Index of notations}\vspace{-4em}

\setglossarystyle{mymcolalttree}

\printglossary[title={\;}]

\end{document}